\documentclass[12pt, a4paper]{article}

\usepackage{latexsym}
\usepackage{amssymb}
\usepackage{amsfonts}
\usepackage{amsthm}
\usepackage{amsmath}
\usepackage[pdftex]{graphicx}

\newcommand{\eps}{{\varepsilon}}

\newtheorem{theorem}{Theorem}[section]
\newtheorem{corollary}[theorem]{Corollary}
\newtheorem{proposition}[theorem]{Proposition}

\DeclareMathOperator{\ee}{\mathbf{e}}

\begin{document}
\title{On  the density of  exponential functionals  of L\'evy processes}
\author{J. C. Pardo\footnote{Centro de Investigaci\'on en Matem\'aticas (CIMAT A.C.), Calle Jalisco s/n, Col. Valenciana, A. P. 402, C.P.
36000, Guanajuato, Gto. Mexico. E-mail: jcpardo@cimat.mx},\hspace{0.7cm}V. Rivero\footnote{Centro de Investigaci\'on en Matem\'aticas (CIMAT A.C.), Calle Jalisco s/n, Col. Valenciana, A. P. 402, C.P.
36000, Guanajuato, Gto. Mexico. E-mail: rivero@cimat.mx}\hspace{0.5cm}and\hspace{0.5cm}K. van Schaik\footnote{School of Mathematics, University of Manchester, Oxford road, Manchester, M13 9PL. United Kingdom. E-mail: Kees.vanSchaik@manchester.ac.uk. This author gratefully acknowledges being supported by a post-doctoral grant from the AXA Research Fund}}
\date{}
\maketitle

\begin{abstract} In this paper, we study the existence of the  density  associated to the  exponential functional of the L\'evy process $\xi$,
\[ 
I_{\ee_q}:=\int_0^{\ee_q} e^{\xi_s} \, \mathrm{d}s, 
\] 
where $\ee_q$ is an  independent exponential r.v.  with parameter $q\geq 0$. In the case when $\xi$ is the negative of a  subordinator,  we prove that the density of $I_{\ee_q}$, here denoted by $k$,  satisfies an integral equation that generalizes the one found by   Carmona et al. \cite{Carmona97}. Finally when $q=0$, we describe explicitly the asymptotic behaviour at $0$ of the density $k$ when $\xi$ is the negative of a subordinator  and at $\infty$ when $\xi$ is a  spectrally positive  L\'evy process that drifts to $+\infty$.

\bigskip

\noindent {\it Keywords:} L\'evy processes, exponential functional, subordinators, self-similar Markov processes.

\medskip

 \noindent{\it AMS 2000 subject classifications: 60G51}

\end{abstract}

\section{Introduction}
\setcounter{equation}{0}

A real-valued L\'evy process is a stochastic process issued from the origin with stationary and independent increments and almost sure right continuous paths with left-limits. We write $\xi=(\xi_t,t\geq 0)$ for its trajectory and $\mathbb{P}$ for its law.  The law $\mathbb{P}$ of a L\'evy processes is characterized by its one-time transition probabilities. In particular there always exists a  triple $(a, \sigma^2, \Pi)$ where $a\in\mathbb{R}$, $\sigma^2\ge 0$ and $\Pi$ is a measure on $\mathbb{R}\backslash\{0\}$ satisfying the integrability condition $\int_{\mathbb{R}}(1\wedge x^2)\Pi({\rm d}x)<\infty$, such that, for $t\geq 0$ and $z \in\mathbb{R}$ 
\begin{equation}
\mathbb{E}[e^{i z \xi_t}] = \exp\left\{ - \Psi(z)t\right\},
\end{equation}
where
\[
\Psi(z) = i a z + \frac{1}{2}\sigma^2 z^2 + \int_{\mathbb{R}}\Big(1- e^{i z x} +i zx\mathbf{1}_{\{|x|<1\}}\Big)\Pi({\rm d}x).
\]
In the case when $\xi$ is a subordinator, 
the L\'evy measure $\Pi$ has support on $[0,\infty)$ and fulfils the extra condition $\int_{(0,\infty)}(1\wedge x)\Pi({\rm d}x)<\infty$. Hence, the characteristic exponent $\Psi$ can be expressed as
\[
\Psi(z)=-icz+\int_{(0,\infty)} \Big(1-e^{i z x}\Big)\Pi({\rm d} x),
\]
where $c\ge 0$ and is known as the drift coefficient. It is well-known that the function $\Psi$ can be extended analytically on the complex upper half-plane, so the Laplace exponent of $\xi$ is given by
\[
\phi(\lambda):=-\log \mathbb{E}[e^{-\lambda\xi_1}]=\Psi(i\lambda)=c\lambda+\int_{(0,\infty)}\Big(1-e^{-\lambda x}\Big)\Pi(\mathrm{d}x).
\]
Similarly, in the case when $\xi$ is a spectrally negative L\'evy process (i.e. has no positive jumps), 
the L\'evy measure $\Pi$ has support on $(-\infty,0)$ and  the characteristic exponent $\Psi$ can be written as
\[
\Psi(z)=i a z + \frac{1}{2}\sigma^2 z^2 + \int_{(-\infty,0)}\Big(1- e^{i z x} +i zx\mathbf{1}_{\{x>-1\}}\Big)\Pi({\rm d}x).
\]
 It is also well-known that the function $\Psi$ can be extended analytically on the complex lower half-plane, so its Laplace exponent satisfies
\[
\psi(\lambda):=\log \mathbb{E}[e^{\lambda\xi_1}]=-\Psi(-i\lambda)= a \lambda + \frac{1}{2}\sigma^2 \lambda^2 + \int_{(-\infty,0)}\Big(e^{\lambda x} -1+-\lambda x\mathbf{1}_{\{x>-1\}}\Big)\Pi({\rm d}x).
\]
In this article, one of our aims is to study the existence of the  density  associated to the  exponential functional
\[ 
I_{\ee_q}:=\int_0^{\ee_q} e^{\xi_s} \, \mathrm{d}s, 
\] 
where $\ee_q$ is an  exponential r.v. independent of the L\'evy process $\xi$ with parameter $q\geq 0$ If $q=0,$ then $\ee_q$ is understood as $\infty$. In this case, we assume that the process $\xi$ drifts towards $-\infty$ since it is a necessary and sufficient condition for the almost sure finiteness of $I:=I_{\infty}$,  see for instance  Theorem 1 in Bertoin and Yor \cite{Bertoin05}.  

Up to our knowledge nothing is known about the existence of the density of $I_{\ee_q}$ when $q>0$. In the case when  $q=0$, the existence of the density of $I$ has been proved by  Carmona et al. \cite{Carmona97} for  L\'evy processes whose jump structure is of finite variation and recently by Bertoin et al. \cite{BertoinLinderMaller} (see Theorem 3.9) for  any  real-valued L\'evy process. In particular when $\xi$ is the negative of a subordinator such that $\mathbb{E}[|\xi_{1}|]<\infty,$  Carmona et al. \cite{Carmona97} (see Proposition 2.1)  proved that the r.v. $I$ has a density, $k$, that  is the unique (up to a multiplicative constant) $L^{1}$ positive solution to the equation 
\begin{equation}\label{CPY_main}
(1-cx)k(x) = \int_x^{\infty} \overline{\Pi}(\log(y/x)) k(y) \, \mathrm{d}y,\qquad x\in(0,1/c),
\end{equation} 
where $\overline{\Pi}(x):=\Pi(x,\infty)$.  

Here, we generalize  the above  equation. Indeed,  we establish an integral equation for the density of $I_{\ee_q}$,   $q\ge 0$, when $\xi$ is the negative of a subordinator and we note that when $q=0$, the condition $\mathbb{E}[|\xi_{1}|]<\infty$ is not essential for the existence of its density and the validity of (\ref{CPY_main}). \\

Another interesting problem is determining the behaviour of  the density of  the exponential functional $I$ at $0$ and  at $\infty$. This problem has been recently studied by Kuznetzov \cite{kuz}  for L\'evy processes with rational Laplace exponent  (at $0$ and at $\infty$), by Kuznetsov and Pardo \cite{KuP}  for hypergeometric L\'evy processes (at $0$ and at $\infty$) and by Patie \cite{Pat} for spectrally negative L\'evy processes (at $\infty$). In most of the applications, it is enough to have estimates of  the tail behaviour $\mathbb{P}(I\le t)$ when $t$ goes to $0$ and/or  $\mathbb{P}(I\ge t)$ when $t$ goes to $\infty$. The tail behaviour $\mathbb{P}(I\le t)$ was studied by Pardo \cite{Pardo06} in the case where  the underlying L\'evy process is spectrally positive and its Laplace exponent is regularly varying at infinity with index $\gamma\in(1,2)$, and by Caballero and Rivero \cite{CR} in the case when $\xi$ is the negative of a subordinator whose Laplace exponent is regularly varying at $0$.  Furthermore, the tail behaviour $\mathbb{P}(I\ge t)$ has been studied in a general setting, see  \cite{CP, MZ, Riv, Riv1} . The second main result  of this paper is related to this problem. Namely, we describe explicitly the asymptotic behaviour at $0$ of the density of $I$ when $\xi$ is a subordinator which in particular implies the behaviour of $\mathbb{P}(I<t)$ near $0$.

The paper is organized as follows: in Section 2 we state our main results, in particular we study the density of $I_{\ee_q}$ and the asymptotic behaviour at $0$ of the density of the exponential functional associated to the negative of a subordinator. Section 3 is  devoted to the proof of the main results and in Section 4, we give some  examples and some numerical results for the density  of $I_{\ee_q}$ when the driving process is the negative of a subordinator.

\section{Main results}
Our first  main result states that $I_{\ee_q}$ has a density, for $q>0$.  Before we establish our first Theorem,  we need to introduce some notation and recall some facts about positive self-similar Markov processes (pssMp) which will be our main tool in this first part. 

Let    $(\xi^\dag_{t}, t\geq 0)$ be the process obtained by killing $\xi$ at an independent exponential time of parameter $q>0$, here denoted by $\ee_{q}.$ The law and the lifetime of $\xi^\dag$ are denoted by $\mathbb{P}^\dag$  and  $\beta$, respectively. 

We first note that
$$
\Big(I,\mathbb{P}^\dag\Big)=\left(\int^{\beta}_{0}\exp\Big\{\xi^\dag_{t}\Big\}\mathrm{d}t,\mathbb{P}^\dag\right)\stackrel{d}{=}\left(\int^{\ee_{q}}_{0}e^{\xi_{t}}\mathrm{d}t, \mathbb{P}\right).
$$ 
For $x\geq 0$ let $\mathbb{Q}_{x}$ be the law of $X^{(x)}$, the positive self-similar Markov process with self-similarity index $1$ issued from $x,$ associated to $\xi^\dag$ via its Lamperti's representation (see \cite{la} for more details on this representation), that is for $x> 0$  
$$X^{(x)}_{t}=\begin{cases}
x\exp\Big\{\xi^\dag_{\tau(t/x)}\Big\},& \text{if}\ \tau(t/x)<\infty\\
0,&\text{if}\ \tau(t/x)=\infty,
\end{cases}, \quad t\geq 0;
$$ where $$\tau(s)=\inf\left\{r>0:\int^{r}_{0}e^{\xi^\dag_{t}}\mathrm{d}t>s\right\},\qquad \inf\{\emptyset\}=\infty,$$ and $0$ is understood as a cemetery state. The process $X^{(x)}$ is a strong Markov process and it fulfills the scaling property, i.e. for $k>0,$ 
$$
\Big(kX^{(x)}_{t/k}, t\geq 0\Big)\stackrel{d}{=}\Big(X^{(kx)}_{t}, t\geq 0\Big).
$$
We denote by $T^{(x)}_{0}:=\inf\{t>0: X^{(x)}_{t}=0\}$, the first hitting time of $X^{(x)}$ at 0. Observe that for $s>0,$ we have the following equivalences, 
\[
\tau(s)<\infty\quad\textrm{ iff }\quad \tau(s)\leq\beta\quad\textrm{ iff }\quad\ s\leq\int^{\beta}_{0}e^{\xi^{\dag}_{t}}\mathrm{d}t.
\]
Hence, it follows from the construction of $X$ that the following equality in law follows
$$\Big(T_{0},\mathbb{Q}_{1}\Big)\stackrel{d}{=}\left(\int^{\ee_{q}}_{0}e^{\xi_{t}}\mathrm{d}t, \mathbb{P}\right).$$ 

Now, we have all the elements to establish our first main result. It concerns the existence of the density  of $I_{\ee_{q}}.$
\begin{theorem}\label{densityq}
Let $q>0$, then the function $$h(t):=q\mathbb{Q}_{1}\left[\frac{1}{X_{t}}\mathbf{1}_{\{t<T_{0}\}}\right],\qquad t\geq 0,$$ is a density for the law of $I_{\ee_{q}}.$
\end{theorem}
\begin{corollary}\label{coro1}
Assume $q>0$ and that $\xi$ is a subordinator. Then the law of the r.v.  $I_{\ee_{q}}$ is a mixture of exponential, that is its law has a density $h$ on $(0,\infty)$ which is completely monotone. Furthermore, $\lim_{t\downarrow 0}h(t)=q$. 
\end{corollary}

In the sequel, we assume that $\xi=-\zeta$ where $\zeta$ is  a  subordinator and we denote by $U_q({\rm d} x)$  the   renewal measure of the killed subordinator $(\zeta_t,t\le \ee_q)$, i.e.
\begin{equation}\label{renewal}
\mathbb{E}\left[\int_0^{\ee_q} f(\zeta_t){\rm d}t\right]=\int_{[0,\infty)} f(x)U_q({\rm d}x),
\end{equation}
where $f$ is a positive measurable function. If the renewal measure is absolutely continuous with respect to the Lebesgue measure, the function $u_q(x)=U_q({\rm d}x)/{\rm d}x$, is usually called the renewal density. If $q=0$, we denote $U_0$ and $u_0$ by $U$ and $u$.

Before stating our first main result, which is a generalization of the integral equation (\ref{CPY_main}) of Carmona et al. for subordinators, in the next proposition we establish that the density of $I_{\ee_q}$ solves an integral equation  in terms of the renewal measure $U_q$.  
\begin{proposition} Let $q\ge 0$. The  random variable $I_{\ee_q}$ has a density that we denote by $k$, and it solves the equation
\begin{equation}\label{prop2}
\int_y^\infty k(x)\mathrm{d}x=\int_0^\infty k(ye^x)U_q(\mathrm{d}x), \qquad \textrm{ almost everywhere.}
\end{equation}
\end{proposition}

The next result generalizes (\ref{CPY_main}).  
\begin{theorem}\label{genCPY} Let $q\ge 0$. The  random variable $I_{\ee_q}$ has a density that we denote by $k$, and it solves
\begin{equation}\label{prop1}
(1-cx)k(x) = \int_x^{\infty} \overline{\Pi}(\log (y/x)) k(y)\mathrm{d}y +q\int_x^\infty k(y)\mathrm{d}y\, \,\qquad x\in(0,1/c). 
\end{equation} 
Conversely, if a density on $(0,1/c)$ satisfies this equation then it is the density of $I_{\ee_q}$.
\end{theorem}

The importance of the above result will be illustrated in Theorem 2.5 where  we study the asymptotic behaviour at $0$ of the density $k$, and in Section 4 where we provide some examples where $k$ can be computed explicitly. Further applications have been provided in Haas \cite{Ha} and Haas and Rivero \cite{HaR} where this equation
has been used to estimate the right tail behavior of the law of $I$ and to study the
maximum domain of attraction of $I.$

The following corollary  is another important application of equation (\ref{prop1}). In parti\-cular, it says that if we know  the density of the exponential functional of the negative of a subordinator, say $k$, 
then for $\rho\ge 0$, $x^\rho k(x)$  adequately normalized is the density of the  exponential functional associated to the negative of a new subordinator. The proof of this fact follows easily by multiplying in both sides of equation (\ref{prop1}) by $x^\rho$. Such result also appears in Chazal et al. \cite{CKP} but in terms of the distribution of $I_{\ee_q}$ not in terms of its density. 

 \begin{corollary} Let $q\ge 0,\rho> 0$, $c_\rho$ a positive constant satisfying
\[
c_\rho=\int_{(0,\infty)}x^{\rho}k(x)\mathrm{d} x, 
\]
and suppose that when $q>0$ the renewal measure $U_q$ has a density. Then the function $h(x):=c^{-1}_{\rho}x^\rho k(x)$ is the density of the exponential functional of the negative of a subordinator whose Laplace exponent is given by 
\begin{equation}\label{trans}
\phi_{\rho}(\lambda)=\frac{\lambda}{\lambda+\rho}\Big(\phi(\lambda +\rho) +q\Big).
\end{equation}
Moreover, the density $h$ solves the equation 
\begin{equation}\label{cor2}
(1-cx)h(x) = \int_x^{\infty} \overline{\Pi}_\rho(\log y/x) h(y)\mathrm{d}y  \, \,\qquad x\in(0,1/c),
\end{equation} 
where $\overline{\Pi}_\rho(z)=\overline{\Pi}(z)e^{-\rho z}+qe^{-\rho z}$.
\end{corollary}

We remark that the transformation studied in Chazal et al. \cite{CKP} is more general than  the one presented in (\ref{trans}) and  that they  applied such transformation to L\'evy processes with one-sided jumps.  We also remark that  the subordinator whose Laplace exponent is given by $\phi_\rho$ has an infinite lifetime in any case.\\

Our next goal is to study the behavior of the density of $I_{\ee_{q}}$ near $0$.  When $q=0,$ we work with the following assumption:\\

\noindent{\bf (A)} {\it The L\'evy measure $\Pi$ belongs to the class $\mathcal{L}_{\alpha}$ for some $\alpha\geq 0,$ that is to say that the tail L\'evy measure $\overline{\Pi}$ satisfies 
\begin{equation}\label{conv_eq_def} 
\lim_{x \to \infty} \frac{\overline{\Pi}(x+y)}{\overline{\Pi}(x)} = e^{-\alpha y}, \quad \mbox{for all $y \in \mathbb{R}$.}
\end{equation}
}
Observe that  regularly varying and subexponential tail L\'evy measures satisfy this assumption with $\alpha=0$ and that convolution equivalent L\'evy measures are examples of L\'evy measures satisfying (\ref{conv_eq_def}) for some index $\alpha>0.$ 

\begin{theorem} Let $q\geq 0$ and $\xi=-\zeta$, where $\zeta$ is a subordinator such that when $q=0$ the L\'evy measure $\Pi$ satisfies assumption (A).  The following asymptotic behaviour holds for the density function $k$ of the exponential functional $I_{\ee_{q}}.$
\begin{itemize}
\item[i)] If $q>0,$ then $$k(x)\xrightarrow[]{} q \qquad \mbox{ as \,$x\downarrow 0$}.$$ 
\item[ii)] If $q=0,$ then $\mathbb{E}[I^{-\alpha}]<\infty$ and
\[ k(x) \sim \mathbb{E}\big[I^{-\alpha}\big] \overline{\Pi}(\log 1/x) \quad \mbox{as \,$x \downarrow 0$.} \]
\end{itemize}

\end{theorem}
In the sequel we will assume that $q=0.$ The above result will help us to describe  the behaviour at $\infty$ of the density of the exponential functional of a particular spectrally negative L\'evy processes associated to the subordinator $\zeta$. In order to explain such relation, we need the following  assumptions. Assume  that $U$, the renewal measure of the subordinator $\zeta$, is absolutely continuous with respect to the Lebesgue measure with density  $u$ which is non-increasing and convex. We also suppose that $\mathbb{E}[\zeta_1]<\infty$. According to Theorem 2 in Kyprianou and Rivero  \cite{KyR} there exists a spectrally negative L\'evy process $Y=(Y_t,t\ge 0)$ that drifts to $+\infty$, whose Laplace exponent is described by 
\[
\psi(\lambda)=\lambda\phi^*(\lambda)=\frac{\lambda^2}{\phi(\lambda)},\qquad \textrm{for }\quad \lambda\ge 0,
\]
where $\phi^*$ is the Laplace exponent of another subordinator and satisfies
\[
\phi^*(\lambda):=q^*+c^*\lambda+\int_{(0,\infty)}\Big(1-e^{-\lambda x}\Big)\Pi^*(\mathrm{d}x),
\]
where 
\[
q^*=\left(c+\int_{(0,\infty)}x\Pi(\mathrm{d}x)\right)^{-1}, \qquad
c^*=\left\{\begin{array}{ll}
0&\textrm{$c>0$, or $\Pi(0,\infty)=\infty$, }\\
1/\Pi(0,\infty), & \textrm{$c=0$ and $\Pi(0,\infty)<\infty$, }
\end{array}\right .
\]
and the L\'evy measure $\Pi^*$ satisfies
\[
u(x)=c^*\mathbf{1}_{\{x=0\}}+q^*+\overline{\Pi}^*(x), \qquad \textrm{for }\quad x\ge 0.
\]
Let $I_\psi$ be the exponential functional associated to $-Y$, i.e.
\[
I_\psi=\int_0^\infty e^{-Y_s}\mathrm{d} s,
\]
and  denote its density by $k_{\psi}$. From the proof of Proposition 4 in Rivero \cite{Riv}  the density $k_{\psi}$ satisfies
\begin{equation}\label{expfuncsn}
k_\psi(x)=q^*\frac{1}{x}k\left(\frac{1}{x}\right), \qquad \textrm{for }\quad x>0.
\end{equation}
The following corollary give us the asymptotic behaviour at $\infty$ of the density of the exponential functional of $-Y$.
 \begin{corollary}  Suppose that $\zeta$ is a subordinator  satisfying assumption (A) and such that its renewal measure  has a density  which is non-increasing and convex and  let $Y$ be its associated spectrally negative L\'evy process  defined as above. Then  the following asymptotic behaviour holds for the density function $k_\psi$,
\[ k_\psi(x) \sim q^* \mathbb{E}\big[I^{-\alpha}\big]\frac{1}{x} \overline{\Pi}(\log x) \quad \mbox{as $x \to \infty$.} \]
\end{corollary}
\section{Proofs}
\setcounter{equation}{0}
\begin{proof}[Proof of Theorem \ref{densityq}]
We  start the proof by showing that the function 
\[
h(t,x):=q\mathbb{Q}_{x}\left[\frac{1}{X_{t}}\mathbf{1}_{\{t<T_0\}}\right],\qquad t\geq 0,\ x>0
\]
is such that 
\begin{equation}\label{den}
\int^{\infty}_{0} h(t,x)\, \mathrm{d} t=1,\qquad \textrm{ for}\quad x>0.
\end{equation}
Then the result  follows from  the identity (\ref{den}) and the fact that 
\[
h(t+s)=q\mathbb{Q}_{1}\Big[h(s,X_{t})\mathbf{1}_{\{t<T_{0}\}}\Big],\qquad\textrm{for}\quad s,t\geq 0,
\]
which is a straightforward consequence of the  Markov property. 

Let us prove   (\ref{den}).  From the definition of $X$ and  the change of variables  $u=\tau(t/x)$, which implies that $ \mathrm{d}u=x^{-1}\exp\{-\xi^\dag_{\tau(t/x)}\} \mathrm{d}t,$ we get
\begin{equation*}
\begin{split}
\int^{\infty}_{0} h(t,x)\mathrm{d}t&=q\int^{\infty}_{0}\,\mathrm{d}t\,\mathbb{E}\bigg[x^{-1}\exp\Big\{-\xi^\dag_{\tau(t/x)}\Big\}\mathbf{1}_{\{\tau(t/x)<\infty\}}\bigg]\\
&=q\mathbb{E}\left[\int^{\infty}_{0}x^{-1}\exp\Big\{-\xi^{\dag}_{\tau(t/x)}\Big\}\mathbf{1}_{\big\{t\leq x\int^{\beta}_{0}e^{\xi^\dag_{s}}\mathrm{d}s\big\}}\mathrm{d}t\right]\\
&=q\mathbb{E}\left[\int^{\infty}_{0}\mathbf{1}_{\{u\leq\beta\}}\mathrm{d}u\right]=q\mathbb{E}(\beta)=1.
\end{split}
\end{equation*}
We now prove that 
\[
\int^{\infty}_{t}h(s)\,\mathrm{d}s=\mathbb{P}\Big(I_{\ee_q}>t\Big),\qquad t>0.
\]
 Indeed, let $t>0$ making a change of variables, using the semi-group property, and Fubini's theorem we have 
 \[
 \int^{\infty}_{t}h(s)\,\mathrm{d}s=\int^{\infty}_{0}h(s+t,1)\,\mathrm{d}s=\mathbb{Q}_{1}\left[\left(\int^{\infty}_{0}h(s,X_{t})\,\mathrm{d}s\right)\mathbf{1}_{\{t<T_{0}\}}\right]=\mathbb{Q}_{1}(t<T_{0}).
 \] 
 The result follows from the identity
$\mathbb{Q}_{1}(t<T_{0})=\mathbb{P}\Big(I_{\ee_q}>t\Big).$
\end{proof}
\begin{proof}[Proof of Corollary \ref{coro1}] Here, we use the same notation as above and we  follow similar arguments as in the proofs of Lemma 5 and Proposition 1 in \cite{Bertoin01}.  We first prove that for every $0\le t<T_0$ and $p>0$, the variable
\[
X_t^{p}\int_t^{T_0}\frac{1} {X^{p+1}_s}\mathrm{d} s 
\]
is independent from $\sigma\{X_s, 0\le s\le t\}$ and is distributed as 
\[
\int_0^{\ee_q}e^{-p\xi_s}\mathrm{d} s.
\]
As a consequence of the Markov property at time $t$, we only need to show that under $\mathbb{Q}_x$, the variable 
\[
x^{p}\int_0^{T_0} \frac{1} {X^{p+1}_s}\mathrm{d} s
\]
is distributed as $\int_0^{\ee_q}e^{-p\xi_s}\mathrm{d} s$. Then the change of variables $t=\tau(s/x)$, $s=x\int_0^t e^{\xi^\dag_u}\mathrm{d}u$, yields
\[
\begin{split}
x^{p}\int_0^{T_0} \frac{1} {X^{p+1}_s} \mathrm{d} s&=x^{-1}\int_0^{T_0}e^{-(p+1)\xi^\dag_{\tau(s/x)}}\mathrm{d} s\\
&=\int_0^{\beta}e^{-(p+1)\xi^\dag_{t}}e^{\xi^\dag_t}\mathrm{d} t\\
&=\int_0^{\beta}e^{-p\xi^\dag_{t}}\mathrm{d} t,
\end{split}
\]
which implies the desired identity in law since $(\xi^\dag_{t}, 0\le t\le \beta)$ and $(\xi_t,0\le t\le \ee_q) $ have the same law. Hence, we have
\[
\mathbb{Q}_1\left[\int_t^{T_0}\frac{1} {X^{p+1}_s}\mathrm{d} s\right]=\frac{\mathbb{Q}_1\Big[X^{-p}_t; t<T_0 \Big]}{\phi(p)+q},
\]
which implies 
\[
\frac{\partial\mathbb{Q}_1\Big[X^{-p}_t; t<T_0 \Big]}{\partial t}=-(\phi(p)+q)\mathbb{Q}_1\Big[X^{-(p+1)}_t; t<T_0 \Big].
\]
By iteration, we get that the function $t\mapsto \mathbb{Q}_1\Big[X^{-p}_t; t<T_0 \Big]$ is completely monotone and takes value 1 for $t=0$. Thus taking $p=1$, we deduce that $h(t)$ is completely monotone on $(0,\infty)$ and that $\lim_{t\downarrow 0} h(t)=q$. Finally  from Theorem 51.6 and Proposition 51.8 in \cite{Sa}, we have that  the law of $I_{\ee_q}$ is a mixture of exponentials.
\end{proof}

\begin{proof}[Proof of Proposition 2.3] The proof follows from the identity
\[
\mathbb{E}\Big[I_{\ee_q}^{n}\Big]=\frac{n}{\phi(n)+q}\mathbb{E}\Big[I_{\ee_q}^{n-1}\Big], \qquad n>0
\]
Indeed, on the one hand it is clear that
\[
\mathbb{E}\Big[I_{\ee_q}^{n}\Big]=\int_0^{\infty}x^n k(x)\mathrm{d}x=n\int_{0}^\infty \mathrm{d}y\, y^{n-1}\int_{y}^{\infty}k(x) \mathrm{d}x.
\]
On the other hand from the identity (\ref{renewal}) with $f(x)=e^{-nx}$ and a  change of variables, we get
\[
\begin{split}
\frac{n}{\phi(n)+q}\mathbb{E}\Big[I_{\ee_q}^{n-1}\Big]&=n\int_0^\infty U_q(\mathrm{d} x)\,e^{-nx}\int_{0}^\infty y^{n-1}k(y)\mathrm{d}y\\
&=n\int_0^\infty U_q(\mathrm{d} x)\int_{0}^\infty y^{n-1}e^{-nx}k(y)\mathrm{d}y\\
&=n\int_0^\infty U_q(\mathrm{d} x)\int_{0}^\infty z^{n-1}k(ze^{-x})\mathrm{d}z\\
&=n\int_0^\infty  \mathrm{d}z\,z^{n-1}\int_{0}^\infty k(ze^{-x})U_q(\mathrm{d} x).
\end{split}
\]
Then putting the pieces together, we have
\[
\int_0^\infty \mathrm{d}y\, y^{n-1}\int_{y}^{\infty}k(x) \mathrm{d}x=\int_0^\infty  \mathrm{d}y\,y^{n-1}\int_{0}^\infty k(ye^{-x})U_q(\mathrm{d} x), \qquad \textrm{for }\,n>0,
\]
which implies the desired result because the density $$y\mapsto\frac{1}{\mathbb{E}(I_{\ee_{q}})}\int_{y}^{\infty}k(x) \mathrm{d}x,$$
is determined by its entire moments, which in turn is an easy consequence of the fact that $k$ is so. 
\end{proof}

\begin{proof}[Proof of Theorem \ref{genCPY}]
By Theorem 2.1 (when $q>0$) and  Theorem 3.9 in \cite{BertoinLinderMaller} (when $q=0$), we know that there exists  a density of $I_{\ee_q}$,  for $q\ge 0$,  that we denote by $h$. Moreover,  in \cite{Carmona97} it has been proved that the moments of $I_{\ee_q}$ are given by 
\begin{equation}\label{moments}
\mathbb{E}\Big[I_{\ee_q}^{n}\Big]=\frac{n!}{\prod^{n}_{i=1}\Big(q+\phi(i)\Big)},\qquad\,\, n\in\mathbb{N}
\end{equation} where the product is understood as $1$ when $n=0.$ 

We first prove that the function $\widetilde{h}:(0,\infty)\to(0,\infty)$ defined via 
$$
\widetilde{h}(x)=\begin{cases}cxh(x)+\displaystyle\int^{\infty}_{x}\overline{\Pi}(\log(y/x))h(y)\mathrm{d}y+q\int_{x}^\infty h(y)\mathrm{d}y ,&\text{if}\ x\in(0,1/c),\\0,& \text{elsewhere},\end{cases}
$$ 
is a density for the law of $I_{\ee_q}$ and hence that $h=\widetilde{h}$ a.e. Then we prove that the equality (\ref{prop1}) holds. In order to do so,  it is enough to verify that 
$$\int^{\infty}_{0}x^{n}\widetilde{h}(x)\mathrm{d}x=\frac{n!}{\prod^{n}_{i=1}\Big(q+\phi(i)\Big)},\qquad n\in\mathbb{N},$$ 
since the law of $I_{\ee_q}$ is determined by its entire moments.

Indeed, elementary computations, identity (\ref{prop2}) and the fact that
\[
\int_{0}^\infty e^{-\theta y}U_q(\mathrm{d}y)=\frac{1}{\phi(\theta)+q},\qquad \theta\ge 0,
\]
give that for any integer  $n\geq 0$,
\begin{equation*}
\begin{split}
\int^{\infty}_{0}x^{n}\widetilde{h}(x)\mathrm{d}x&=c\int^{\infty}_{0}\mathrm{d}x\, x^{n+1}h(x)+\int^{\infty}_{0}\mathrm{d}x\,x^{n}\int^{\infty}_{x}\mathrm{d}y\,\overline{\Pi}(\log(y/x))h(y)\\
 &\hspace{6.5cm} +q\int_0^\infty \mathrm{d} x x^n\int_{0}^\infty h(xe^y)U_q(\mathrm{d}y)\\
&=\frac{n!(n+1)c}{\prod^{n+1}_{i=1}\Big(q+\phi(i)\Big)}+\int^{\infty}_{0}\mathrm{d}y\,h(y)\int^{y}_{0}\mathrm{d}x\,x^{n}\overline{\Pi}(\log(y/x))\\
&\hspace{6.5cm} +q\int_{0}^\infty U_q(\mathrm{d}y)\int_0^\infty \mathrm{d} x x^n h(xe^y)\\
&=\frac{n!(n+1)c}{\prod^{n+1}_{i=1}\Big(q+\phi(i)\Big)}+\int^{\infty}_{0}\mathrm{d}y\,h(y)y^{n+1}\int^{\infty}_{0}\mathrm{d}z\,e^{-(n+1)z}\overline{\Pi}(z)\\
&\hspace{6.2cm} +q\int_{0}^\infty  U_q(\mathrm{d}y)e^{-(n+1)y}\int_0^\infty \mathrm{d} z z^n h(z)\\
&=\frac{n!(n+1)c}{\prod^{n+1}_{i=1}\Big(q+\phi(i)\Big)}+\frac{(n+1)!}{\prod^{n+1}_{i=1}\Big(q+\phi(i)\Big)}\frac{\displaystyle\int^{\infty}_{0}\big(1-e^{-(n+1)z}\big)\Pi(\mathrm{d}z)}{n+1}\\
&\hspace{5.7cm} +q\frac{n!}{\prod^{n}_{i=1}\Big(q+\phi(i)\Big)}\int_{0}^\infty U_q(\mathrm{d}y)e^{-(n+1)y}\\
&=\frac{n!}{\prod^{n}_{i=1}\Big(q+\phi(i)\Big)}\frac{(n+1)c+\displaystyle\int^{\infty}_{0}\big(1-e^{-(n+1)z}\big)\Pi(\mathrm{d}z)+q}{q+\phi(n+1)}\\
&=\frac{n!}{\prod^{n}_{i=1}\Big(q+\phi(i)\Big)}.
\end{split}
\end{equation*}
Now, let $\mathcal{N}=\left\{x\in\mathbb{R}: h(x)\neq\widetilde{h}(x)\right\}.$ By the above arguments, we know that the Lebesgue measure of $\mathcal{N}$ is zero. Let $k:(0,\infty)\to(0,\infty)$ be the function defined by 
$$
k(x)=\begin{cases}h(x),& \text{if}\ x\in\mathcal{N}^{c},\\
\displaystyle\frac{1}{1-cx}\left(\displaystyle\int^{\infty}_{x}\overline{\Pi}(\log(y/x))h(y)\mathrm{d}y+q\int_x^\infty h(y)\mathrm{d}y\right),& \text{if}\ x\in\mathcal{N}.\end{cases}
$$  
We now prove  that $k(x)$ satisfies  equation (\ref{prop1}) everywhere. If $x\in\mathcal{N}^{c}$ then  we have that $k(x)=h(x)=\widetilde{h}(x),$ and hence  equation (\ref{prop1}) is verified. Finally, if $x\in\mathcal{N},$ we have the following equalities 
\begin{equation*}
\begin{split}
cxk(x)&+\int^{\infty}_{x}\overline{\Pi}(\log(y/x))k(y)\mathrm{d}y+q\int_x^\infty k(y)\mathrm{d} y\\
&=cxk(x)+\int^{\infty}_{x}\overline{\Pi}(\log(y/x))k(y)\mathbf{1}_{\{y\in\mathcal{N}^{c}\}}\mathrm{d}y+q\int_x^\infty k(y)\mathbf{1}_{\{y\in\mathcal{N}^{c}\}}\mathrm{d} y\\
&=cxk(x)+\int^{\infty}_{x}\overline{\Pi}(\log(y/x))h(y)\mathbf{1}_{\{y\in\mathcal{N}^{c}\}}\mathrm{d}y+q\int_x^\infty h(y)\mathbf{1}_{\{y\in\mathcal{N}^{c}\}}\mathrm{d} y\\
&=\frac{cx}{1-cx}\left(\int^{\infty}_{x}\overline{\Pi}(\log(y/x))h(y)\mathrm{d}y+q\int_x^\infty h(y)\mathrm{d} y) \right)\\
&\hspace{4.5cm}+\int^{\infty}_{x}\overline{\Pi}(\log(y/x))h(y)\mathrm{d}y+q\int_x^\infty h(y)\mathrm{d} y\\ 
&=k(x).
\end{split}
\end{equation*}  
Conversely if $k$ is a density on $(0,1/c)$ satisfying equation (\ref{prop1}), from the above computations it is clear that $k$ and $I_{\ee_q}$ have the same entire moments. This implies that the  $k$ is a density of   the exponential functional $I_{\ee_q}$.
\end{proof}
\begin{proof}[Proof of Theorem 2.6] The proof consists of three steps. First we show that when $q=0,$ $\mathbb{E}\big[I^{-\alpha}\big]<\infty$, then for $q\geq 0$ we obtain a technical estimate on the maximal growth of $k(x)$ as $x \downarrow 0$, and finally the statement of the theorem.

\emph{Step 1}. Here we assume that $q=0$ and prove that $\mathbb{E}\big[I^{-\alpha}\big]<\infty.$ The case $\alpha=0$ is obvious. For $\alpha \in (0,1)$, we have from Theorem 2 in \cite{Bertoin05} that there exists a random variable $R$, independent of $\xi$, such that $IR \stackrel{d}{=} \mathbf{e}$, where $\mathbf{e}$ follows a unit mean exponential distribution. Since $\mathbb{E}[\mathbf{e}^{-\alpha}]<\infty$, the result follows. 

Finally let $\alpha \geq 1$. With (\ref{prop1}) and some standard computations, we find
\[
\begin{split}
\int_0^\infty x^{-\beta-1} k(x) \, \mathrm{d}x &=  c \int_0^\infty \mathrm{d}x\,x^{-\beta} k(x) + \int_0^\infty \mathrm{d}x \, x^{-\beta-1} \, \int_x^\infty \mathrm{d}y\,\overline{\Pi}(\log (y/x)) k(y)  \nonumber\\
 & =  c \mathbb{E}\Big[I^{-\beta}\Big] + \int_0^\infty \mathrm{d}y \, k(y) \int_0^y \mathrm{d} x\,x^{-\beta-1} \overline{\Pi}(\log (y/x))  \nonumber\\
 &=  c \mathbb{E}\Big[I^{-\beta}\Big] + \int_0^\infty\mathrm{d}y\, y^{-\beta} k(y) \int_0^\infty\mathrm{d}u\, e^{\beta u} \overline{\Pi}(u) \nonumber\\
 & = -\frac{1}{\beta} \mathbb{E}\Big[I^{-\beta}\Big] \left( -c\beta + \int_0^\infty \Big(1-e^{\beta z}\Big) \, \Pi(\mathrm{d}z) \right), \nonumber
\end{split}
\]
that is to say 
\begin{equation}\label{KvS5} 
\mathbb{E}\Big[I^{-\beta-1}\Big]=\mathbb{E}\Big[I^{-\beta}\Big] \frac{\phi(-\beta)}{-\beta}, 
\end{equation}
where $\phi$ is the Laplace exponent of $\xi$, which can be extended to $(-\alpha,\infty)$ since for $\beta<\alpha$
\begin{equation}\label{KvS20}
 \int_0^\infty ( e^{\beta u}-1) \, \Pi(\mathrm{d}u) = \beta \int_1^\infty \overline{\Pi}(\log (z)) z^{\beta-1} \, \mathrm{d}z < \infty.
\end{equation}
To see that (\ref{KvS20}) holds, note that $\overline{\Pi}(\log (z))$ is regularly varying with index $-\alpha$ by (\ref{conv_eq_def}). Hence $\overline{\Pi}(\log z)=z^{-\alpha} \ell(z)$ for a slowly varying function $\ell$ and we can apply Proposition 1.5.10 from Bingham et al. \cite{Bingham87}.

Now, by iteratively using (\ref{KvS5}) we see that for $\mathbb{E}[I^{-\alpha}]<\infty$ it is enough to have $\mathbb{E}[I^{-\alpha'}]<\infty$ for some $\alpha' \in [0,1)$. But this obviously holds if $\alpha'=0$, while if $\alpha' \in (0,1)$ it holds by the same argument as we used above for the case $\alpha \in (0,1)$.

\emph{Step 2}. We assume that $q\geq 0.$ For $q=0,$ let $p$ be any function such that $p(0)=0$ and $\min \{ \alpha-1,0 \} < p(\alpha) < \alpha$ for all $\alpha >0$. When $q>0$ the function $p$ will be taken as zero and hence the symbol $p(\alpha)$ will be taken as $0$. The goal of this step is to show 

\begin{equation}\label{result_step2}
 \frac{k(x)}{x^{p(\alpha)}}\qquad \mbox{ stays bounded as $\qquad x \downarrow 0$.} 
\end{equation}
Observe that when $q>0$ it follows from (\ref{prop1}) that $\liminf_{x\to0}k(x)\geq q.$
Set $h(x):=k(x)/x^{p(\alpha)}$. We can write (\ref{prop1}) as

\begin{equation}\label{KvS10}
1-cx = x \int_1^\infty \overline{\Pi}(\log (z)) z^{p(\alpha)} \frac{h(xz)}{h(x)} \, \mathrm{d}z+ \frac{qx^{p(\alpha)}\mathbb{P}(I_{\ee_{q}}>x)}{h(x)}.
\end{equation} 
We argue by contradiction. Take some $\hat{x} \in (0,1/c)$. If $h$ were not bounded at $0+$, then $\mathbf{1}_{\{ x \leq \hat{x} \}} h(x)$ would keep on attaining new maxima as $x \downarrow 0$. (Note that $\hat{x}$ is present just to make sure this statement also holds if $k$ is not bounded at $1/c-$.) In particular this means that a sequence of points $(x_n)_{n \geq 0}$ exists with $x_n \downarrow 0$ as $n \to \infty$ and such that $h(x_n) \geq \sup_{x \in [x_n,\hat{x}]} h(x)$. We will show that this implies 
\[
x_n \int_1^\infty \overline{\Pi}(\log (z)) z^{p(\alpha)} \frac{h(x_n z)}{h(x_n)} \,  \mathrm{d}z +\frac{qx^{p(\alpha)}_{n}\mathbb{P}(I_{\ee_{q}}>x_{n})}{h(x_{n})}\to 0\quad\textrm{ as }\quad n \to \infty,
\]
 which indeed contradicts (\ref{KvS10}) since $1-c x_n \to 1$ as $n \to \infty$. Observe that if $q>0$ and $h$ is not bounded at $0+$ then the second term in the latter equation tends to $0,$ because $p(\alpha)=0,$ by construction. So we just have to prove that the first term in the latter equation tends to $0.$ For this, we have
\begin{equation}\label{KvS11}
\begin{split}
x_n \int_1^\infty \overline{\Pi}(\log (z)) z^{p(\alpha)} \frac{h(x_n z)}{h(x_n)} \, \mathrm{d}z & = x_n \int_1^{\hat{x}/x_n} \overline{\Pi}(\log (z)) z^{p(\alpha)} \frac{h(x_n z)}{h(x_n)} \, \mathrm{d}z \\
 &  \hspace{.5cm} + \, x_n \int_{\hat{x}/x_n}^{\infty} \overline{\Pi}(\log (z)) z^{p(\alpha)} \frac{h(x_n z)}{h(x_n )} \, \mathrm{d}z. 
\end{split}
\end{equation}
We first deal with the first integral on the right hand side of (\ref{KvS11}). By construction of the sequence $(x_n)_{n \geq 0}$, we have $h(x_n z)\leq h(x_n) $ for any $z \in [1,\hat{x}/x_n]$, hence  

\begin{equation}\label{KvS13}
  x_n \int_1^{\hat{x}/x_n} \overline{\Pi}(\log (z)) z^{p(\alpha)} \frac{h(x_n z)}{h(x_n)} \, \mathrm{d}z\le x_n \int_1^{\hat{x}/x_n} \overline{\Pi}(\log (z)) z^{p(\alpha)} \, \mathrm{d}z.
\end{equation}
If $q>0$ or $\alpha=0$ (recall $p(0)=0$), we can take any $1<z_0$ and write
\[
\begin{split}
 x_n \int_1^{\hat{x}/x_n} \overline{\Pi}(\log(z)) \, \mathrm{d}z &= x_n \int_1^{z_0} \overline{\Pi}(\log (z)) \, \mathrm{d}z + x_n \int_{z_0}^{\hat{x}/x_n} \overline{\Pi}(\log (z)) \, \mathrm{d}z \\
&\leq x_n \int_1^{z_0} \overline{\Pi}(\log (z)) \, \mathrm{d}z + x_n \left( \frac{\hat{x}}{x_n} -z_0 \right)  \overline{\Pi}(\log (z_0)), 
\end{split}
\]
where the inequality uses that $\overline{\Pi}$ is decreasing. Letting $n \to \infty$, recalling that $x_n \downarrow 0$, we see that the first integral on the right hand side vanishes while the second term tends to $\hat{x} \overline{\Pi}(\log z_0)$. As we can make this term arbitrarily small by choosing $z_0$ large enough, since $\overline{\Pi}(\log z) \to 0$ as $z \to \infty$, it follows indeed that (\ref{KvS13}) vanishes. 

Next, let $\alpha >0$. Since $\alpha-1<p(\alpha)<\alpha$, we can choose some $\beta \in (0,\alpha)$ such that $p(\alpha)-\beta+1 \in (0,1)$. Using this we find
\[
\begin{split} 
x_n \int_1^{\hat{x}/x_n} \overline{\Pi}(\log (z)) z^{p(\alpha)} \, \mathrm{d}z &= x_n \int_1^{\hat{x}/x_n} \overline{\Pi}(\log (z)) z^{\beta-1} z^{p(\alpha)-\beta+1} \,\mathrm{d}z \\
&\leq x_n \left( \frac{\hat{x}}{x_n} \right)^{p(\alpha)-\beta+1} \int_1^{\hat{x}/x_n} \overline{\Pi}(\log (z)) z^{\beta-1} \,\mathrm{d}z ,
\end{split}
\]
and the right hand side indeed vanishes as $n \to \infty$, again since $x_n \downarrow 0$ and by (\ref{KvS20}).


It remains to show that the second integral on the right hand side of (\ref{KvS11}) vanishes as $n \to \infty$. We have
\[
\begin{split} 
x_n \int_{\hat{x}/x_n}^{\infty} \overline{\Pi}(\log (z)) z^{p(\alpha)} \frac{h(x_n z)}{h(x_n )} \, \mathrm{d}z &\leq x_n \overline{\Pi}(\log (\hat{x}/x_n)) \frac{1}{h(x_n)} \int_{\hat{x}/x_n}^{\infty} z^{p(\alpha)} h(x_n z) \, \mathrm{d}z \\
&= \frac{\overline{\Pi}(\log (\hat{x}/x_n))}{x_n^{p(\alpha)}}  \frac{1}{h(x_n)} \int_{\hat{x}}^\infty k(u) \, \mathrm{d}u,
\end{split}
\]
where the inequality uses that $\overline{\Pi}$ is decreasing and the equality the definition of $h$ together with the substitution $u=x_n z$. Since $k$ is a density and by assumption $h(x_n) \to \infty$ as $n$ goes to $\infty$, for the right hand side to vanish it remains to show that $\overline{\Pi}(\log \hat{x}/x_n)/x_n^{p(\alpha)}$ stays bounded as $n$ increases. If $q>0$ or $\alpha=0$ (recall $p(0)=0$) it is immediate since $\overline{\Pi}$ is decreasing. If $\alpha >0$, for any $1<z_0<z$ integration by parts yields
\[ 
\overline{\Pi}(\log (z)) z^{p(\alpha)} = p(\alpha) \int_{z_0}^z \overline{\Pi}(\log (u)) u^{p(\alpha)-1} \, \mathrm{d}u + \int_{z_0}^z u^{p(\alpha)} \, \mathrm{d} \overline{\Pi}(\log (u)) + \overline{\Pi}(\log (z_0)) z_0^{p(\alpha)}. 
\]
Now, if we let $z$ goes to $\infty$, then since $p(\alpha) < \alpha$ we see from (\ref{KvS20}) that the first integral in the right hand side stays bounded while the second integral is negative on account of the fact that $\overline{\Pi}$ is decreasing. Consequently the left hand side has to stay bounded and we are done.

\emph{Step 3, case $q=0$}.  Denote $C_\alpha=\mathbb{E}[I^{-\alpha}]$, which is finite by Step 1. From (\ref{prop1}) we obtain for all $x>0$,
\begin{equation}\label{CPY_div}
 (1-cx)\frac{k(x)}{\overline{\Pi}(\log (1/x))} = \int_x^{\infty} \frac{\overline{\Pi}(\log (y/x))}{\overline{\Pi}(\log (1/x))} k(y) \, \mathrm{d}y. 
 \end{equation}
Using this equation together with $k \geq 0$, Fatou's lemma and the fact that $\Pi$ has an exponential tail (cf. (\ref{conv_eq_def})) yields
\[
\begin{split} 
\liminf_{x \downarrow 0} \frac{k(x)}{\overline{\Pi}(\log (1/x))} &=  \liminf_{x \downarrow 0} \frac{cxk(x)}{\overline{\Pi}(\log (1/x))}+\liminf_{x \downarrow 0} \int_x^{\infty} \frac{\overline{\Pi}(\log (y/x))}{\overline{\Pi}(\log (1/x))} k(y) \,  \mathrm{d}y \\
&\geq \int_0^\infty y^{-\alpha} k(y) \, \mathrm{d}y = C_\alpha. 
\end{split}
\]
On the other hand, for any $\eps>0$ we have as $x \downarrow 0$

\[ \int_\eps^{\infty} \frac{\overline{\Pi}(\log (y/x))}{\overline{\Pi}(\log (1/x))} k(y) \,  \mathrm{d}y \to \int_\eps^{\infty} y^{-\alpha} k(y) \,  \mathrm{d}y \leq C_\alpha. \]
If $\alpha>0$, this follows from the fact that the convergence (\ref{conv_eq_def}) is uniform over $y \in [\eps,\infty)$, see e.g. Theorem 1.5.2 in \cite{Bingham87}. If $\alpha=0$ this uniformity holds only over intervals of the form $[\eps,x_0]$, in which case we can write the left hand side as the sum of integrals over $[\eps,x_0]$ and $[x_0,\infty)$, the former in the limit again is bounded above by $C_\alpha,$ while for the latter we can use that $\overline{\Pi}$ is decreasing to see

\[ \int_{x_0}^{\infty} \frac{\overline{\Pi}(\log(y/x))}{\overline{\Pi}(\log (1/x))} k(y) \,  \mathrm{d}y \leq \frac{\overline{\Pi}(\log (x_0/x))}{\overline{\Pi}(\log (1/x))} \int_{x_0}^{\infty} k(y) \,  \mathrm{d}y, \]
then letting first $x \to \infty$, thereby using (\ref{conv_eq_def}), and then $x_0 \to \infty$ it follows that this term vanishes.

So it remains to show that 

\[ \limsup_{x \downarrow 0} \int_x^{\eps} \frac{\overline{\Pi}(\log (y/x))}{\overline{\Pi}(\log (1/x))} k(y) \,  \mathrm{d}y \to 0 \quad \mbox{as $\eps \to 0$.} \]
For this, we get for $\eps$ small enough and $x<\eps$:
\[
\begin{split}  
\frac{1}{\overline{\Pi}(\log (1/x))} \int_x^{\eps} \overline{\Pi}(\log (y/x)) k(y) \,  \mathrm{d}y &= \frac{x}{\overline{\Pi}(\log (1/x))} \int_1^{\eps/x} \overline{\Pi}(\log (z)) k(xz) \,  \mathrm{d}z \\
&\leq \frac{C x}{\overline{\Pi}(\log (1/x))} \int_1^{\eps/x} \overline{\Pi}(\log (z)) (xz)^{p(\alpha)} \,  \mathrm{d}z\\
& = \frac{C x^{1+p(\alpha)}}{\overline{\Pi}(\log (1/x))} \int_1^{\eps/x} \overline{\Pi}(\log (z)) z^{p(\alpha)} \,  \mathrm{d}z \\ 
&\sim \frac{C' x^{1+p(\alpha)}}{\overline{\Pi}(\log (1/x))} \left( \frac{\eps}{x} \right)^{p(\alpha)+1} \overline{\Pi}(\log (\eps/x)) \quad \mbox{as $x \downarrow 0$,}
\end{split}
\]
where $C$ and $C'$ are constants, the inequality holds by Step 2 (cf. (\ref{result_step2})) and the asymptotics follows from Karamata's theorem (see e.g. Theorem 1.5.11 in \cite{Bingham87}), which indeed applies here since $\overline{\Pi}(\log (z))$ is regularly varying with index $-\alpha$ (cf. (\ref{conv_eq_def})) and by construction (see Step 2) $p(\alpha) \geq \alpha-1$. Now, using (\ref{conv_eq_def}) we see that the ultimate right hand side goes to $C' \eps^{p(\alpha)+1-\alpha}$ as $x \downarrow 0$ and this indeed vanishes as $\eps \to 0$ since by construction $p(\alpha)+1-\alpha>0$ for all $\alpha \geq 0$. 

\noindent\emph{Step 3, case $q>0$}. We will prove that $$\int^{\infty}_{x}\overline{\Pi}(\log(y/x))k(y)dy\xrightarrow[x\to 0]{}0.$$ By Step 2, we can assume that $k$ is bounded by $K\geq q,$ in a neighborhood of $0+.$ Let $\delta>1$ fixed, for $x$ small enough we have that 
\begin{equation*}
\begin{split}
\int^{x\delta}_{x}\overline{\Pi}(\log(y/x))k(y)\mathrm{d}y&\leq K \int^{x\delta}_{x}\overline{\Pi}(\log(y/x))\mathrm{d}y\\
&=K \int^{\log\delta}_{0}\overline{\Pi}(u)xe^{u}\mathrm{d}u\\
&\leq Kx\delta\int^{\log\delta}_{0}\overline{\Pi}(u)\mathrm{d}u\xrightarrow[x\to0]{}0.
\end{split}
\end{equation*}
Also, we have that 
\begin{equation*}
\begin{split}
\int^{\infty}_{x\delta}\overline{\Pi}(\log(y/x))k(y)\mathrm{d}y\leq \overline{\Pi}(\log\delta)\int^{\infty}_{x\delta}k(y)\mathrm{d}y\xrightarrow[x\to 0]{}\overline{\Pi}(\log\delta),
\end{split}
\end{equation*}
We conclude by making $\delta\to\infty.$ Indeed, using equation (\ref{prop1}) and the above arguments we conclude that 
$$(1-cx)k(x)-q\mathbb{P}(I_{\ee_q}>x)\xrightarrow[x\to 0]{}0,$$ from where the
result follows.

\end{proof}

\section{Examples and some numerics}
In this section, we illustrate Theorem 2.3 , Corollary 2.4 and equation (\ref{expfuncsn}) with some examples and we also provide some applications of Theorem 2.5.\\

\noindent{\bf Example 1.}  Let $q>0$ and consider the  case when the subordinator is just a linear drift with $c>0$. By a simple Laplace inversion, we deduce $u_q( x)=c^{-1}e^{-\frac{q}{c}x}$. Thus, from identities  (\ref{prop1}) and (\ref{prop2}) we get
\[
(1-cx)k(x) = \frac{q}{c}\int_{[0,\infty)} k(xe^y)e^{-\frac{q}{c}y} \mathrm{d} y,\, \,\qquad x\in(0,1/c).
\]
After straightforward computations, we deduce that the density of $I_{\ee_q}$ is of the form
\[
k(x)=q(1-cx)^{\frac{q}{c}-1},\, \,\qquad x\in(0,1/c).
\]

Let $\rho>0$ and note that 
\[
\phi_\rho(\theta)=c\theta+ q\frac{\theta}{\theta+\rho}\qquad\textrm{and}\qquad c_\rho=\frac{q}{c^{\rho+1}}\frac{\rho(\rho+1)\Gamma(q/c)}{\Gamma(\rho+q/c+1)}.
\]
According to Corollary 2.4, the density of the exponential functional of the subordinator whose Laplace exponent is given by  $\phi_\rho$, satisfies
\[
h(x)=c^{\rho+1}\frac{\Gamma(\rho+q/c+1)}{\Gamma(\rho+1)\Gamma(q/c)}x^{\rho}(1-cx)^{q/c-1},\qquad\textrm{ for }\quad x\in(0,1/c),
\]
in other words  the exponential functional has the same law as $c^{-1}B(\rho+1, q/c)$, where $B(\rho+1, q/c)$ is a beta r. v. with parameters $(\rho+1, q/c)$.

Now, let us consider the associated spectrally negative Levy process $Y$ whose Laplace exponent is written as follows
\[
\psi(\lambda)=\frac{\lambda^2}{\phi_\rho(\lambda)}=\frac{\lambda(\lambda+\rho)}{c(\lambda+\rho)+q},
\]    
From  (\ref{expfuncsn}), we deduce that the density of the exponential functional $I_\psi$ associated to $Y$ satisfies
\[
k_\psi(x)=\frac{\rho c^{\rho+1}}{c\rho+q}\frac{\Gamma(\rho+q/c+1)}{\Gamma(\rho+1)\Gamma(q/c)}x^{-(\rho+q/c)}(x-c)^{q/c-1}\qquad\textrm{ for }\quad x> c.
\]
Hence $I_\psi$ has the same law as $c(B(\rho, q/c))^{-1}$.\\

\noindent{\bf Example 2.} Let  $q=c=0$, $\beta>0$  and 
\[
\overline{\Pi}(z)=\frac{\beta}{\Gamma(a+1)}e^{-\frac{(s-1)}{a} z}\Big(e^{\frac{z}{a}}-1\Big)^{a-1},
\] 
where $a\in (0,1]$ and $s\ge a$. Thus, the Laplace exponent $\phi$ has the form
\[
\phi(\theta)=\beta\frac{\theta\Gamma(a(\theta-1)+s)}{\Gamma(a\theta+s)}.
\]
In this case, the equation (\ref{prop1}) can be written as follows
\[
\begin{split}
k(x)&=\frac{\beta}{\Gamma(a+1)}\int_x^\infty (y/x)^{-\frac{s-1}{a}}\left(\left(y/x\right)^{\frac{1}{a}}-1\right)^{a-1}k(y)\mathrm{d} y\\
&=\frac{\beta x}{\Gamma(a)}\int_0^\infty (z+1)^{a-s}z^{a-1}k\Big(x(z+1)^a\Big)\mathrm{d} z,
\end{split}
\]
where we are using the change of variable $z=(y/x)^{\frac{1}{a}}-1$. After some computations we deduce  that 
\begin{equation}\label{example2}
k(z)=\frac{\beta^{s/a}}{a\Gamma(s)}z^{\frac{s-a}{a}} e^{-(\beta z)^{\frac{1}{a}}}, \qquad\textrm{ for }\quad z\ge 0.
\end{equation}
In other words $I$ has the same law as $\beta^{-1}\gamma_s^a$, where $\gamma_s$ is a gamma r.v. with parameter $s$. 

If $a=1$, the process  $\xi$ is a compound Poisson process of parameter   $\beta>0$ with exponential jumps of mean $(s-1)^{-1}>0$. From (\ref{example2}), it is clear that the law of its associated exponential functional has the same law as $\gamma_{(s,\beta)}$,  a gamma r.v. with parameters $(s,\beta)$.

We now consider the associated spectrally negative Levy process $Y$ whose Laplace exponent satisfies 
\[
\psi(\lambda)=\frac{\lambda^2}{\phi(\lambda)}=\frac{\lambda\Gamma(a\lambda+s)}{\beta\Gamma(a(\lambda-1)+s)}.
\]
The density of the exponential functional $I_\psi$ associated to $Y$ is given by
\[
k_\psi(x)=\frac{\beta^{\frac{s-a}{a}}}{a\Gamma(s-a)}x^{-s/a}e^{-(\beta/x)^{1/a}},\qquad x>0.
\]
We remark that when $a=1$, the process $\xi$ is a Brownian motion with drift and that the exponential functional $I_\psi$  has the same law as $\gamma_{(s-1,\beta)}^{-1}$. This identity in law has been established by  Dufresne \cite{Du}.

Next, let $\rho>0$ and note that 
\[
\phi_\rho(\theta)=\beta\frac{\theta\Gamma(a(\theta+\rho-1)+s)}{\Gamma(a(\theta+\rho)+s)}\qquad\textrm{and}\qquad c_\rho=\frac{\Gamma(a\rho+s)}{\beta^\rho\Gamma(s)}.
\]
According to Corollary 2.4, the density of the exponential functional of the subordinator whose Laplace exponent is given by  $\phi_\rho$, satisfies
\[
h(x)=\frac{\beta^{(s+a\rho)/a}}{a\Gamma(a\rho+s)}x^{(a\rho+s-a)/a}e^{- (\beta x)^{1/a}}\qquad\textrm{ for }\quad x>0,
\]
i.e. it has the same law   as $\beta^{-1}\gamma_{a\rho +s}^a$. In particular,  the density of the exponential functional of its associated spectrally negative Levy process satisfies
\[
k_\psi(x)=\frac{\beta^{(s+a\rho-a)/a}}{a\Gamma(a(\rho-1)+s)}x^{-(a\rho+s)/a}e^{- (\beta/x)^{1/a}},\quad x>0.
\]

\noindent{\bf Example 3.} Finally, let $a\in(0,1)$, $\beta\ge a$, $c=0$, $q=\Gamma(\beta)/\Gamma(\beta-a)$,
\[
\overline{\Pi}(z)=\frac{1}{\Gamma(1-a)}\int_z^{\infty}\frac{e^{(1+a-\beta)x/a}}{(e^{x/a}-1)^{1+a}}\mathrm{d} x\quad \textrm{ and }\quad u_q(z) =\frac{1}{\Gamma(a+1)}e^{-(\beta-1)z/a}\Big(e^{\frac{z}{a}}-1\Big)^{a-1}.
\]
The process $\xi$ with such characteristics is a killed Lamperti subordinator with parameters $(1/\Gamma(1-a), 1+a-\beta, 1/a, a)$, see Section 3.2 in Kuznetsov et al. \cite{KKPS} for a proper definition. 
From Theorem 1.3 the density of $I_{\ee_q}$ satisfies the equation
\[
k(x)=\int_0^\infty \left(\frac{xe^y}{\Gamma(1-a)}\int_z^{\infty}\frac{e^{(1+a-\beta)x/a}}{(e^{x/a}-1)^{1+a}}\mathrm{d} x+\frac{\Gamma(\beta)e^{-(\beta-1)z/a}}{\Gamma(\beta-a)\Gamma(a+1)}\Big(e^{\frac{y}{a}}-1\Big)^{a-1}\right)k(xe^y)\mathrm{d}y.
\] 
Since the above equation seems difficult to solve, we use the method of moments in order to determine the law of $I_{\ee_q}$. We first note that
\[
\mathbb{E}\Big[I^n_{\ee_q}\Big]=\frac{n!\Gamma(\beta)}{\Gamma(a n+\beta)}, 
\]
and that in the case $\beta=1$, the exponential functional $I_{\ee_q}$ has the same distribution as   
 $X_a^{-a}$, where $X_a$ is a stable random variable, i.e.
\[
\mathbb{E}\Big[e^{-\lambda X_a}\Big]=\exp\{-\lambda^{a}\}, \qquad \lambda\ge 0,
\]
see Section 3 in \cite{Bertoin01}. Recall that the negative moments of $X_a$ are given by
\[
\mathbb{E}\Big[X_a^{-n}\Big]=\frac{\Gamma(1+n/a)}{\Gamma(1+n)}, \qquad n\ge 0.
\]
Now we introduce $L_{(a,\beta)}$ and $A$, two independent r.v. whose laws are described as follows,
\[
\mathbb{P}(L_{(a,\beta)}\in \mathrm{d} y)=\mathbb{E}\left[\frac{a\Gamma(\beta)}{\Gamma(\beta/a)X^\beta_a};\frac{1}{X_a^a}\in \mathrm{d} y\right],
\]
and 
\[
\mathbb{P}(A\in \mathrm{d} y)=\Big(\beta/a-1\Big)(1-x)^{\beta/a-2}\mathbf{1}_{[0,1]}(x)\mathrm{d} x.
\]
It is important to note from example 1, that $A$ has the same law as the exponential functional associated to the subordinator $\sigma$ which is defined as follows
\[
\sigma_t=t+\beta/a-1, \qquad t\ge 0.
\]
On the one hand, it is clear that 
\[
\mathbb{E}\Big[L_{(a,\beta)}^{n}\Big]=\frac{a\Gamma(\beta)}{\Gamma(\beta/a)}\mathbb{E}\Big[X_a^{-(a n+\beta)}\Big]=\frac{\Gamma(\beta)}{\Gamma(\beta/a)}\frac{\Gamma(n+\beta/a)}{\Gamma(a n+\beta)}, 
\]
and on the other hand, we have
\[
\mathbb{E}\Big[A^n\Big]=\frac{\Gamma(n+1)\Gamma(\beta/a)}{\Gamma(n+\beta/a)},
\]
which implies the $I_{\ee_q}$ has the same law as $L_{a,\beta}A$.\\

Finally we numerically illustrate the density $k$ and its asymptotic behaviour at $0$ for some particular subordinators $\zeta$.  Let us first shortly discuss the method we used. Clearly the equation (\ref{CPY_main}) motivates the following straightforward discretisation procedure: approximate $k$ by a step function $\tilde{k}$, i.e.

\[ \tilde{k}(x) = \sum_{i=0}^{N-1} \mathbf{1}_{\{ x \in  [x_i,x_{i+1})\}} y_i, \] 
where $0=x_0<x_1<\ldots<x_{N}=1/c$ forms a grid on the $x$-axis. The heights $y_i$ can then be found by iterating over $i=N-1,\ldots,0$, thereby at each step using (\ref{CPY_main}) with $x=x_i$ and $k$ replaced by $\tilde{k}$. Two remarks are in place here.

Firstly, as (\ref{CPY_main}) is linear in $k$ the condition that $k$ is a density is required to uniquely determine the solution. This translates to the fact that the numerical procedure discussed above requires a starting point, i.e. the value $y_{N-1}>0$ should be known. (Of course, starting with $y_{N}=0$ yields $\tilde{k} \equiv 0$.) We proceed by first leaving $y_{N-1}$ undetermined, run the iteration so that every $y_i$ in fact becomes a linear function of $y_{N-1}$, and then find $y_{N-1}$ by requiring that $\tilde{k}$ integrates to $1$.

The second remark is that even though any choice of grid would in principle work, we found one in particular to be useful. Indeed, if we set $x_n=(1/c) \Delta^{N-n}$ for some $\Delta$ less than (but typically very close to) $1$, equation (\ref{prop1}) yields the following relation:
\begin{multline*}
(1-cx_n) y_n = \int_{x_n}^{\infty} \overline{\Pi}(\log (y/x_n)) \tilde{k}(y) \, \mathrm{d}y = x_n \int_{1}^{\infty} \overline{\Pi}(\log (z)) \tilde{k}(x_n z) \, \mathrm{d}z \\
= x_n \sum_{i=n}^{N-1} y_i \int_{1}^{\infty} \overline{\Pi}(\log (z)) \mathbf{1}_{\{ x_n z \in [x_i,x_{i+1}) \}} \, vz = x_n \sum_{i=n}^{N-1} y_i \int_{\Delta^{n-i}}^{\Delta^{n-i-1}} \overline{\Pi}(\log (z)) \, \mathrm{d}z.
\end{multline*}
The approximation this setup yields is very efficient in comparison with e.g. the approximation using a standard equidistant grid, due to the fact that in this case we need to evaluate only $N$ different integrals numerically\footnote{All computations were done in the open source computer algebra system SAGE: www.sagemath.org}. 

First we consider two examples for which the density $k$ of $I$ is explicitly known. The first one is taken from Example 2  with $a=1$, $\beta=2$ and $s=3/2$. In this case from (\ref{example2}), we have
\[ 
k(x) = \frac{2^{5/2}}{\sqrt{\pi}} x^{1/2} e^{-2x} \quad \mbox{for $x>0$.} 
\]
See Figures 1-4 for plots of the density $k$, the difference $\tilde{k}-k$ (where $\tilde{k}$ is obtained by the above method with $\Delta=0.998$, yielding a grid of $\approx 4500$ points and a few minutes computation time on an average laptop), the ratio $k(x)/\overline{\Pi}(\log (1/x))$ and the ratio $\tilde{k}(x)/\overline{\Pi}(\log (1/x))$ respectively.
 
The second explicit example is taken again from Example 2 with $\beta=1$ and $s=1$ and $a= 1/2$. In this case from (\ref{example2}), we have
\[ 
k(x) = 2 x e^{-x^2} \quad \mbox{for $x>0$.} 
\]
It is important to note that $\overline{\Pi}$ satisfies ${\bf (A)}$ with $\alpha=1$. In this case, Figures 5 -8 show plots plots of the density $k$, the difference $\tilde{k}-k$, the ratio $k(x)/\overline{\Pi}(\log (1/x))$ and the ratio $\tilde{k}(x)/\overline{\Pi}(\log(1/x))$ respectively.

Next we look at two examples where no formula for $k$ is available. The first one is when $\xi$ is a stable subordinator with drift, i.e. $c=1$ and $\Pi(\mathrm{d}x) = x^{-1-a} \, \mathrm{d}x$, where we take $a=1/4$. See Figures 9 \& 10 for a plot of $\tilde{k}$ and the ratio $\tilde{k}(x)/\overline{\Pi}(\log (1/x))$ respectively. Note that this is an example of a L\'evy measure satisfying (\ref{conv_eq_def}) with parameter $0$.

Finally, the second example is a subordinator $\xi$ with zero drift and L\'evy measure of the form $\Pi(\mathrm{d}x)=x^{-1/4} \exp(-x^n) \, \mathrm{d}x$. Figure 11 shows $\tilde{k}$ for $n=1$ (blue), $n=2$ (purple) and $n=3$ (green) respectively. Figure 12 shows the ratio $\tilde{k}(x)/\overline{\Pi}(\log 1/x)$ for the case $n=1$, since then ${\bf (A)}$ is satisfied with $\alpha=1$.

\clearpage

\begin{figure}
\begin{minipage}[b]{0.5\linewidth}
\centering
\includegraphics[width=8cm]{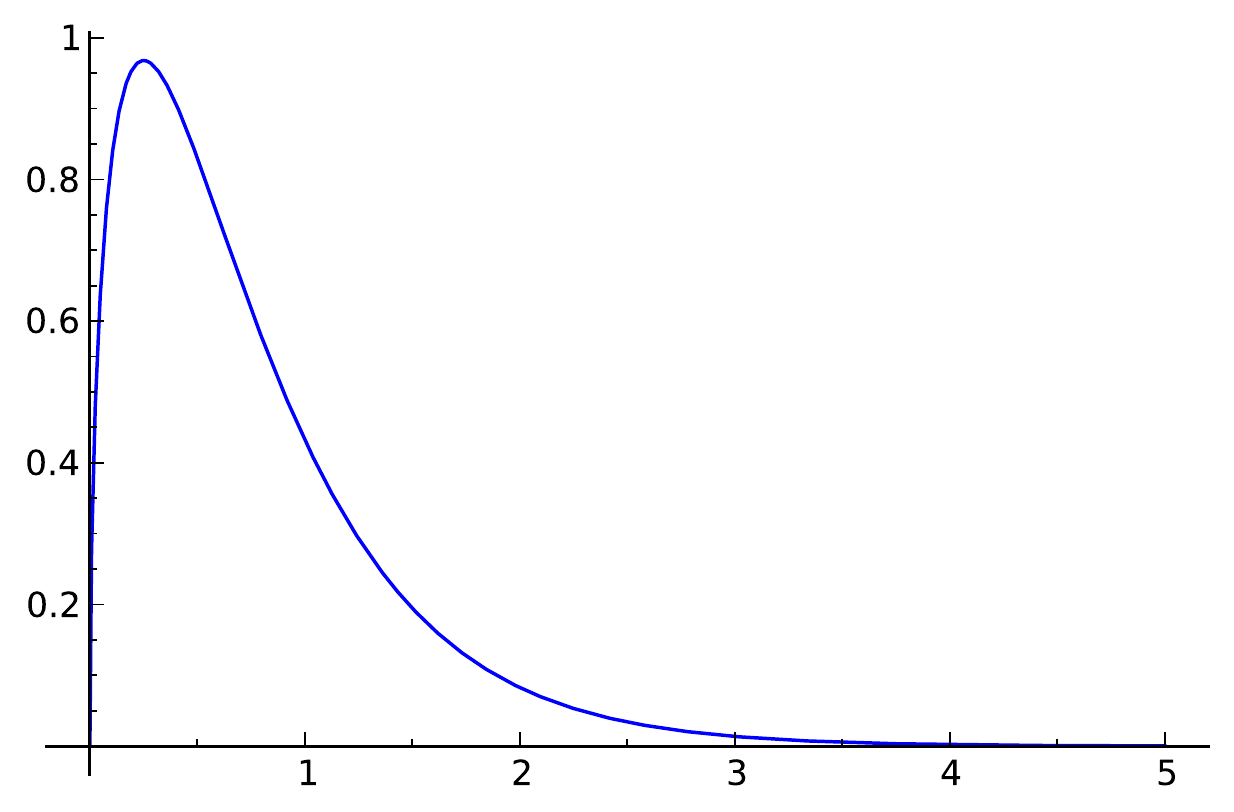}
\caption{The density function $k$}
\end{minipage}
\hspace{0.1cm}
\begin{minipage}[b]{0.5\linewidth}
\centering
\includegraphics[width=8cm]{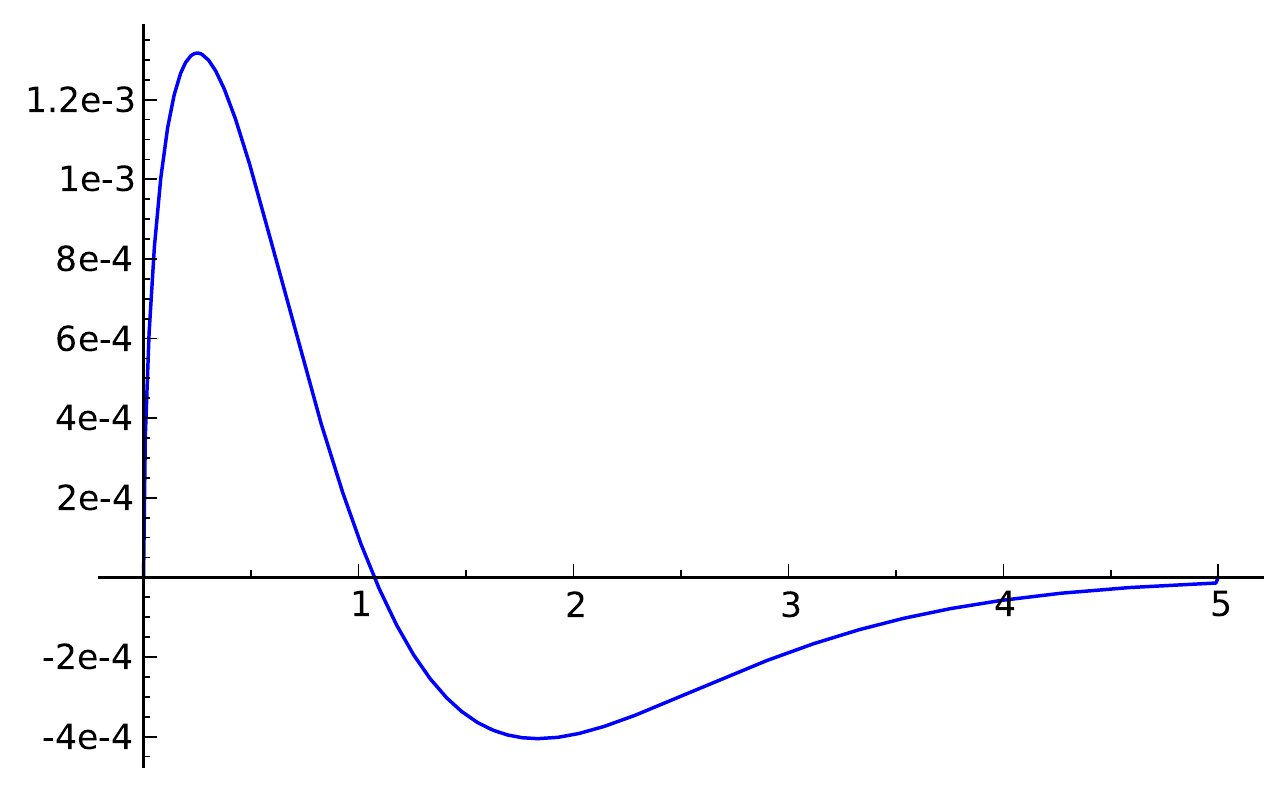}
\caption{The difference $\tilde{k}-k$}
\end{minipage}
\end{figure}

\begin{figure}
\begin{minipage}[b]{0.5\linewidth}
\centering
\includegraphics[width=8cm]{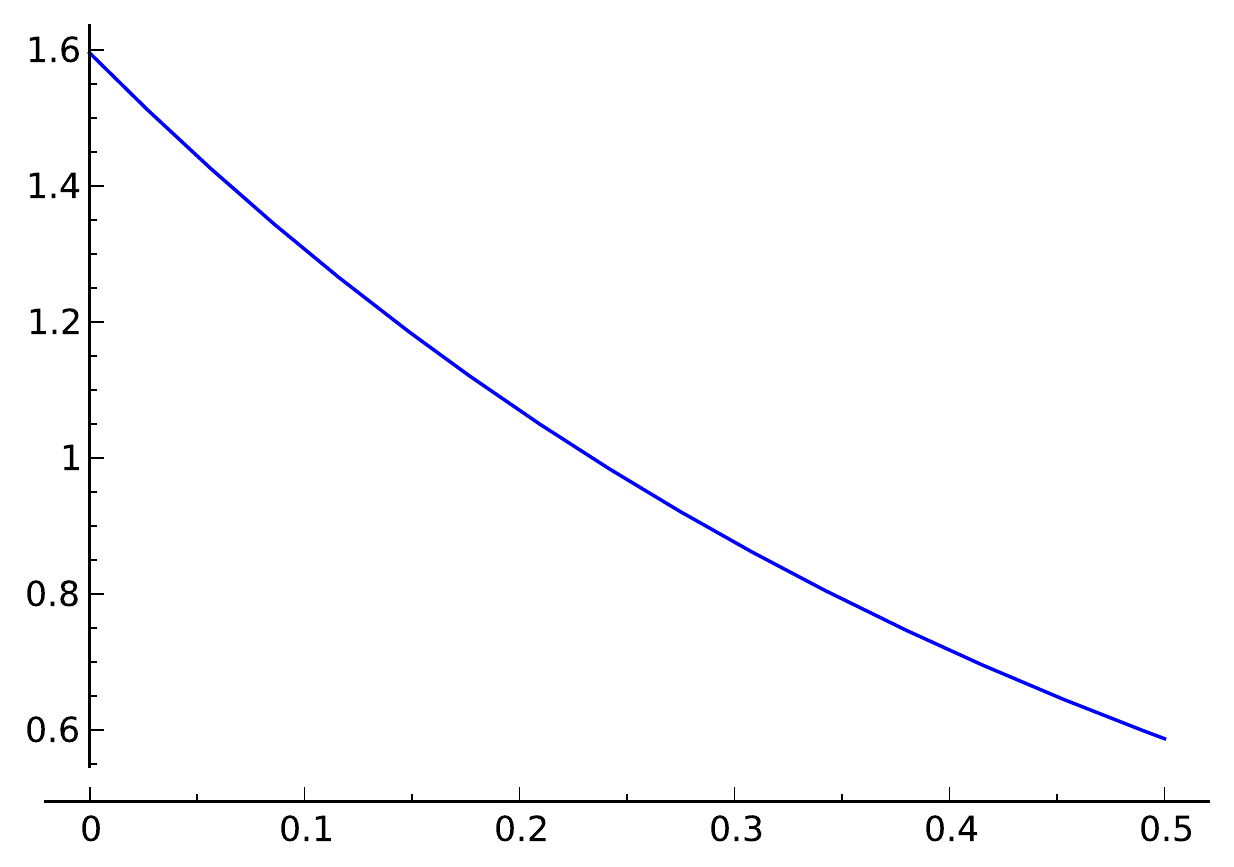}
\caption{The ratio $k(x)/\overline{\Pi}(\log 1/x)$}
\end{minipage}
\hspace{0.1cm}
\begin{minipage}[b]{0.5\linewidth}
\centering
\includegraphics[width=8cm]{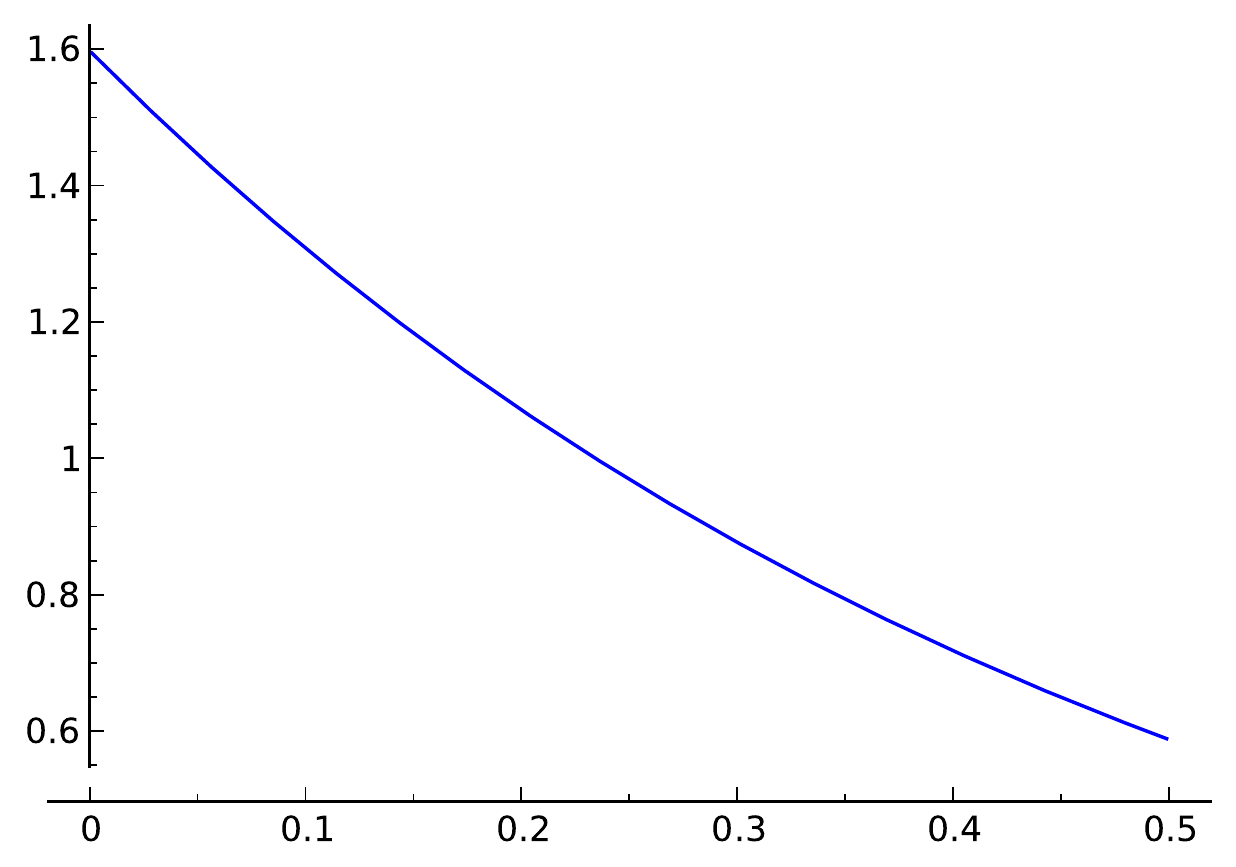}
\caption{The ratio $\tilde{k}(x)/\overline{\Pi}(\log 1/x)$}
\end{minipage}
\end{figure}

\begin{figure}
\begin{minipage}[b]{0.5\linewidth}
\centering
\includegraphics[width=8cm]{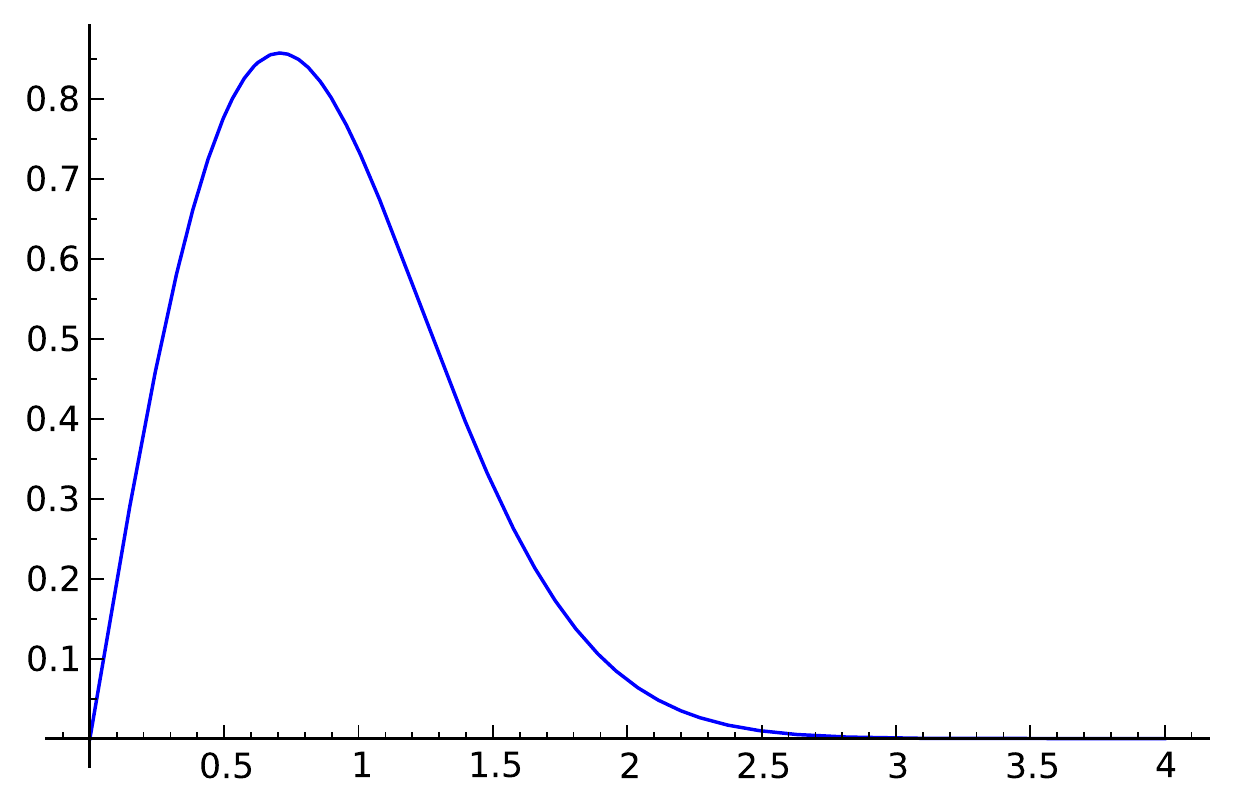}
\caption{The density function $k$}
\end{minipage}
\hspace{0.1cm}
\begin{minipage}[b]{0.5\linewidth}
\centering
\includegraphics[width=8cm]{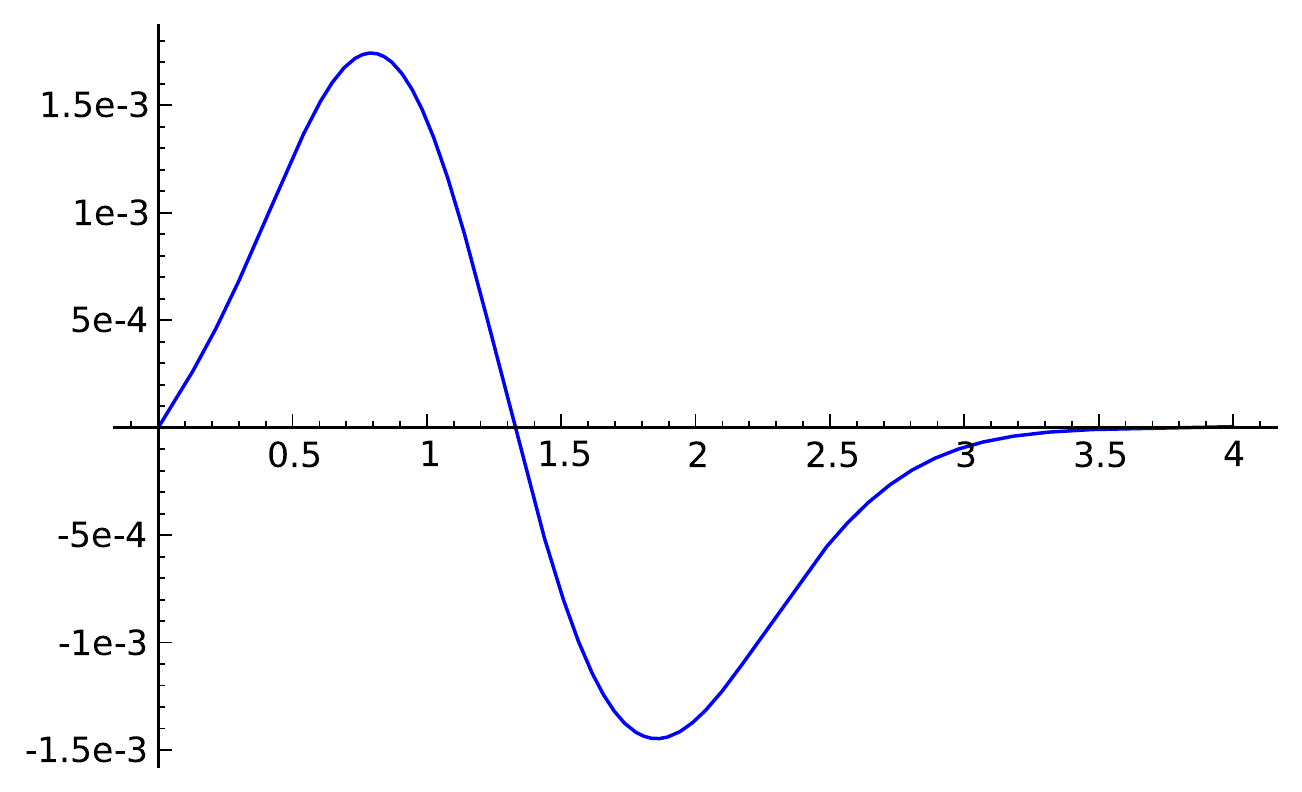}
\caption{The difference $\tilde{k}-k$}
\end{minipage}
\end{figure}

\begin{figure}
\begin{minipage}[b]{0.5\linewidth}
\centering
\includegraphics[width=8cm]{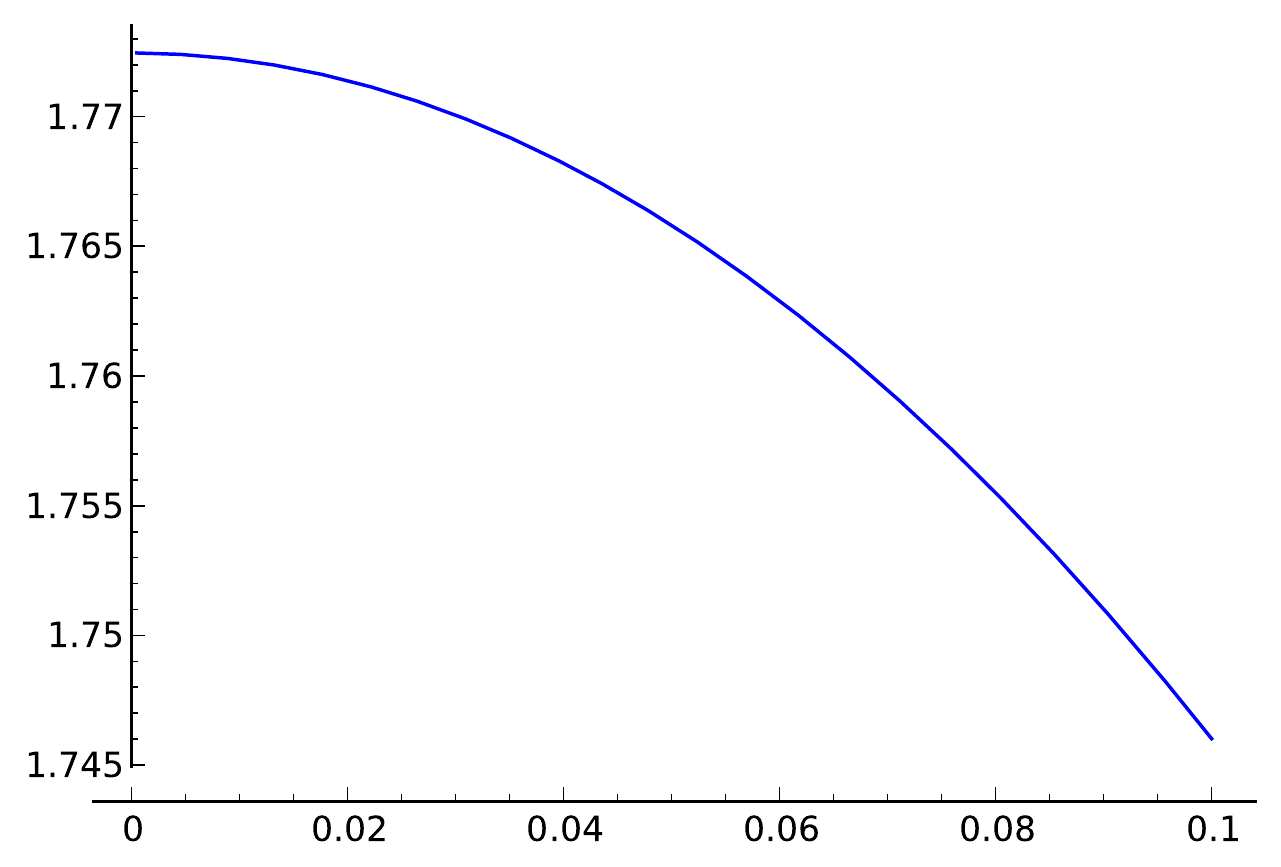}
\caption{The ratio $k(x)/\overline{\Pi}(\log 1/x)$}
\end{minipage}
\hspace{0.1cm}
\begin{minipage}[b]{0.5\linewidth}
\centering
\includegraphics[width=8cm]{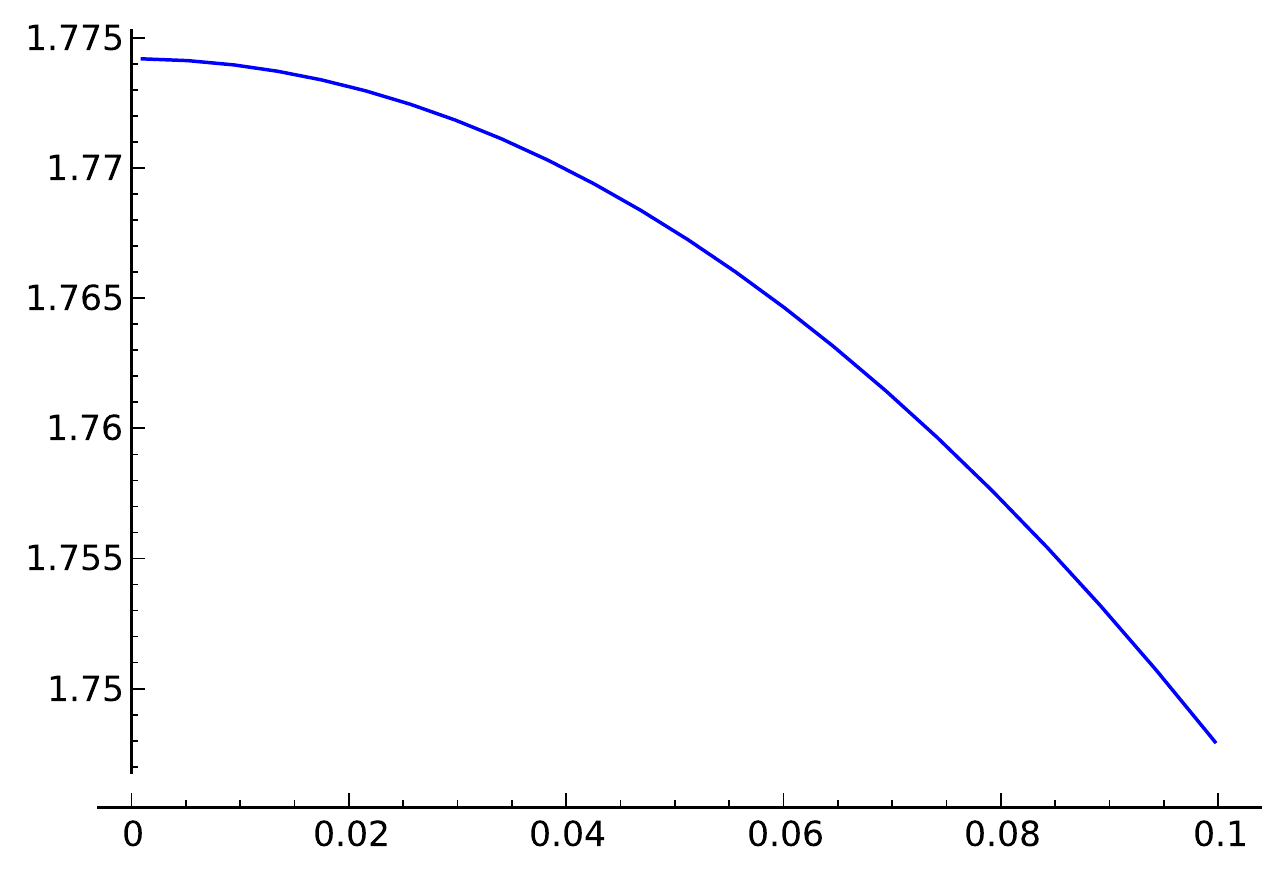}
\caption{The ratio $\tilde{k}(x)/\overline{\Pi}(\log 1/x)$}
\end{minipage}
\end{figure}

\begin{figure}
\begin{minipage}[b]{0.5\linewidth}
\centering
\includegraphics[width=8cm]{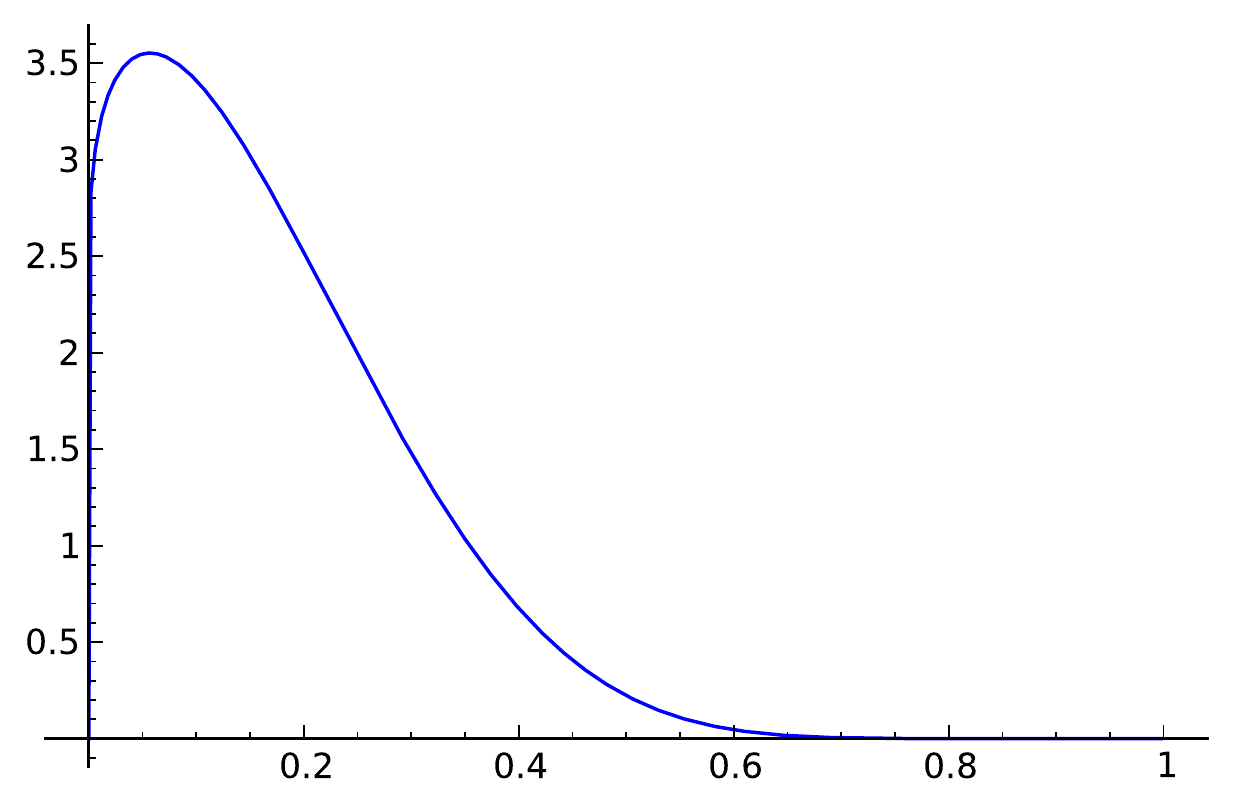}
\caption{The density function $\tilde{k}$}
\end{minipage}
\hspace{0.1cm}
\begin{minipage}[b]{0.5\linewidth}
\centering
\includegraphics[width=8cm]{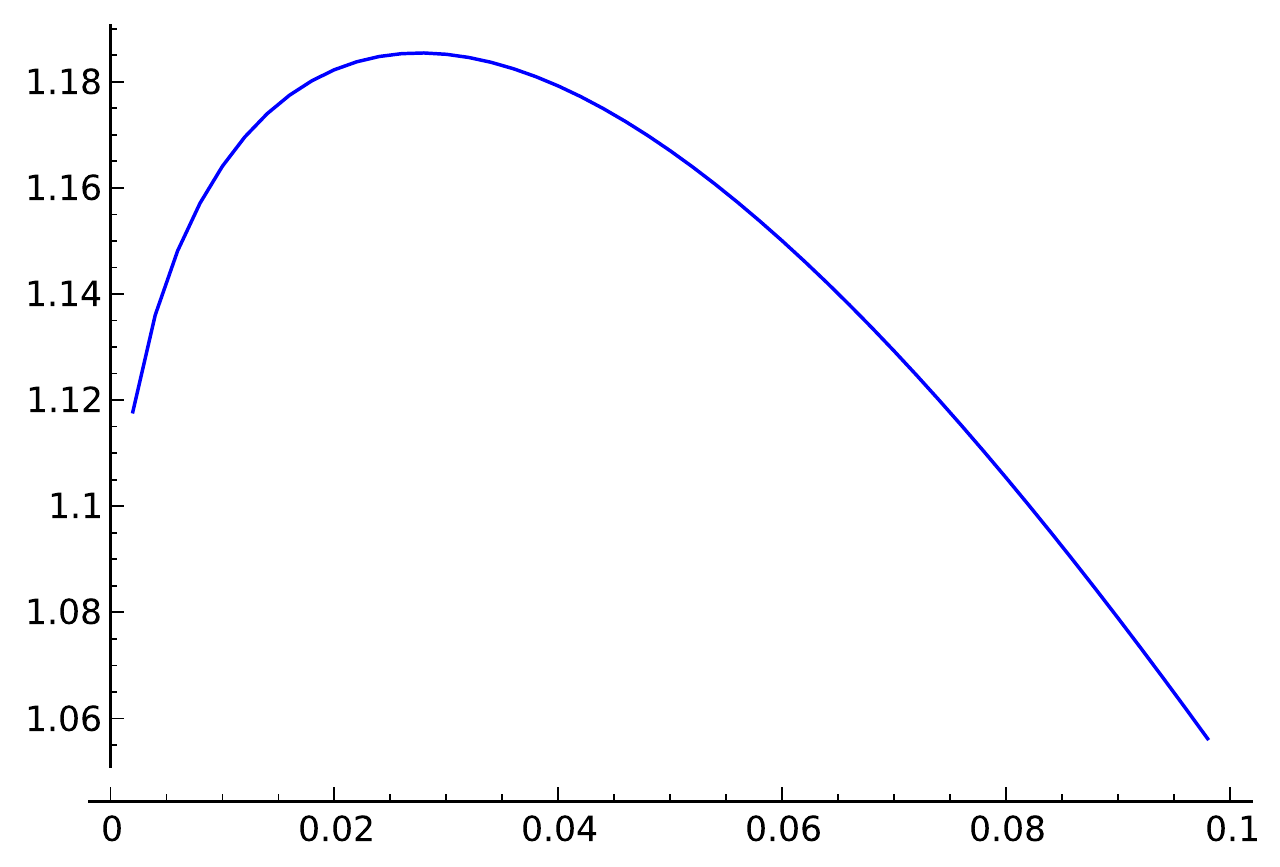}
\caption{The ratio $\tilde{k}(x)/\overline{\Pi}(\log 1/x)$}
\end{minipage}
\end{figure}

\begin{figure}
\begin{minipage}[b]{0.5\linewidth}
\centering
\includegraphics[width=8cm]{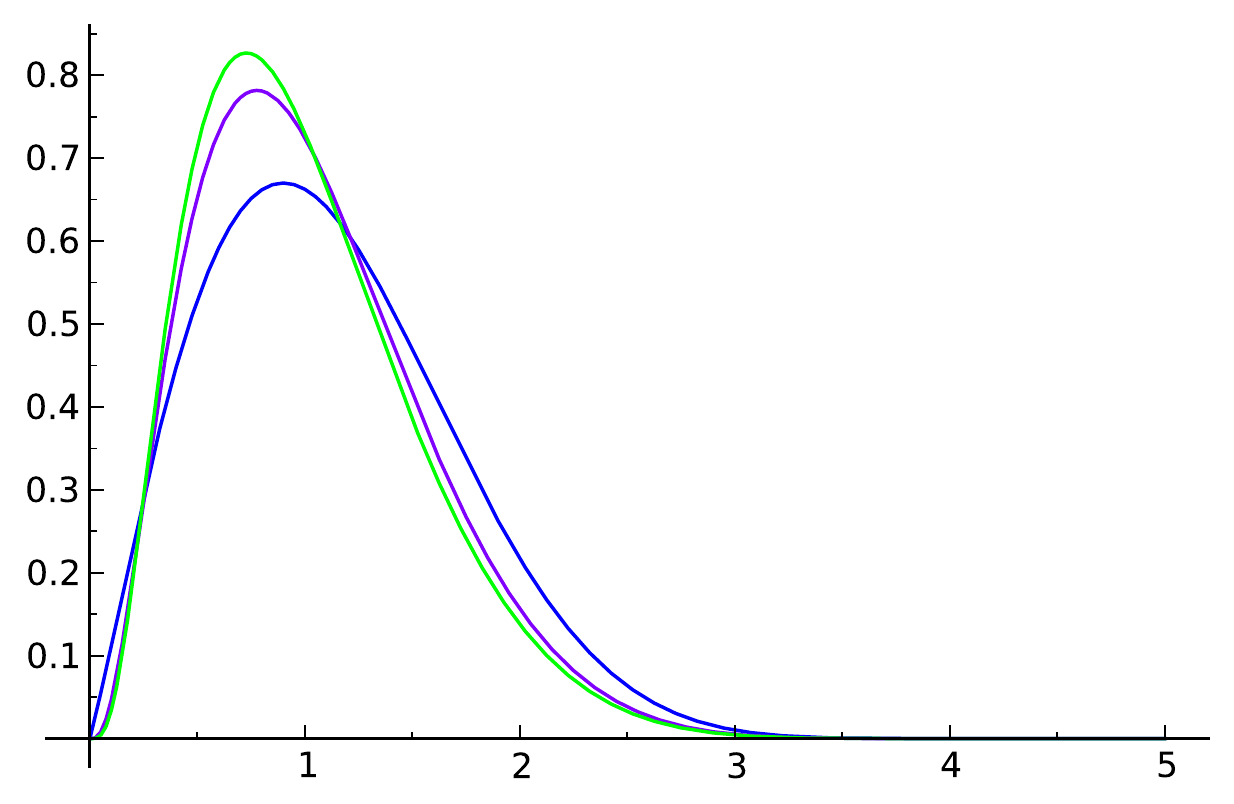}
\caption{Density functions $\tilde{k}$}
\end{minipage}
\hspace{0.1cm}
\begin{minipage}[b]{0.5\linewidth}
\centering
\includegraphics[width=8cm]{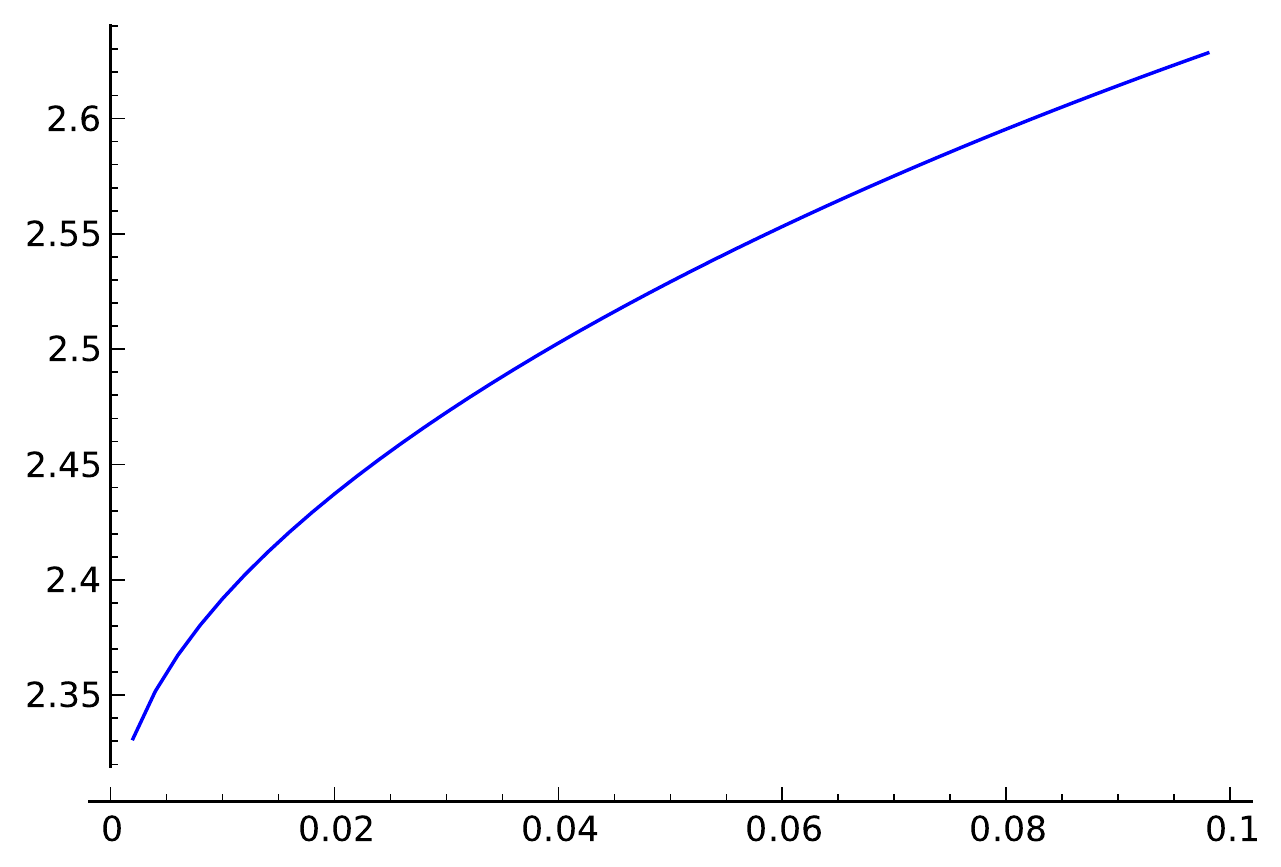}
\caption{The ratio $\tilde{k}(x)/\overline{\Pi}(\log 1/x)$}
\end{minipage}
\end{figure}

\clearpage


\end{document}